\documentclass[final,leqno]{siamltex704}
\usepackage{amsmath}
\usepackage{amssymb}
\usepackage{graphicx}
\usepackage[notcite,notref]{showkeys}
\newcommand{\bv}{{\bf v}}

\newcommand{\bn}{{\bf n}}

\newcommand{\pT}{{\partial T}}
\def\bbQ{\mathbb{Q}}
\def\T{{\mathcal T}}
\def\E{{\mathcal E}}

\def\W{{\mathcal W}}
  \def\b#1{\mathbf{#1}}
\def\a#1{\begin{align*}#1\end{align*}}\def\p#1{\begin{pmatrix}#1\end{pmatrix}}
\def\an#1{\begin{align}#1\end{align}}
\def\l{{\langle}}
\def\r{{\rangle}}\def\d{\operatorname{div}\,}
\def\3bar{{|\hspace{-.02in}|\hspace{-.02in}|}}
\newtheorem{defi}{Definition}[section]

\newtheorem{remark}{Remark}[section]
\newtheorem{algorithm}{Weak Galerkin Algorithm}

\title{ A $C^0$-Weak Galerkin Finite Element
Method for the Biharmonic Equation}

\author{
Lin Mu\thanks{Department of Mathematics, Michigan State University,
      East Lansing, MI 48824 (linmu@ msu.edu)}
\and
Junping Wang\thanks{Division of Mathematical Sciences, National Science
  Foundation, Arlington, VA 22230 (jwang@\break nsf.gov). The research
  of Wang was supported by the NSF IR/D program, while working at
  National Science Foundation. However, any opinion, finding, and
  conclusions or recommendations expressed in this material are those
  of the author and do not necessarily reflect the views of the
  National Science Foundation,}
\and
Xiu Ye\thanks{Department of
  Mathematics, University of Arkansas at Little Rock, Little Rock, AR
  72204 (xxye@ualr.edu). This research was supported in part by
  National Science Foundation Grant DMS-1115097.}
\and
Shangyou Zhang\thanks{Department of Mathematical Sciences
University of Delaware, Newark, DE 19716 (szhang\@\ udel.edu)}}

\begin{document}

\maketitle

\begin{abstract}
A  $C^0$-weak Galerkin (WG) method is introduced and analyzed for
solving the biharmonic equation in 2D and 3D. A weak Laplacian is
defined for $C^0$ functions in the new weak formulation. This WG
finite element formulation is symmetric, positive definite and
parameter free. Optimal order error estimates are established in
both a discrete $H^2$ norm and the $L^2$ norm, for the weak Galerkin
finite element solution. Numerical results are presented to confirm
the theory. As a technical tool, a refined Scott-Zhang interpolation
operator is constructed to assist the corresponding error estimate.
This refined interpolation preserves the volume mass of order
$(k+1-d)$ and the surface mass of order $(k+2-d)$ for the $P_{k+2}$
finite element functions in $d$-dimensional space.
\end{abstract}


\pagestyle{myheadings}

\section{Introduction}

We consider the biharmonic equation of the form
\begin{eqnarray}
\Delta^2 u&=&f,\quad \mbox{in}\;\Omega,\label{pde}\\
u&=&g,\quad\mbox{on}\;\partial\Omega,\label{bc-d}\\
\frac{\partial u}{\partial
n}&=&\phi,\quad\mbox{on}\;\partial\Omega,\,\label{bc-n}
\end{eqnarray}
where $\Omega$ is a  bounded polygonal or
polyhedral domain in $\mathbb{R}^d$ for $d=2, 3$.
For the biharmonic problem (\ref{pde}) with Dirichlet and Neumann boundary conditions (\ref{bc-d}) and (\ref{bc-n}), the corresponding
variational form is given by seeking $u\in H^2(\Omega)$ satisfying
$u|_{\partial \Omega}=g$ and $\frac{\partial u}{\partial
n}|_{\partial \Omega}=\phi$ such that
\begin{equation}\label{wf}
(\Delta u, \Delta v) = (f, v) \qquad \forall v\in H_0^2(\Omega),
\end{equation}
where $H_0^2(\Omega)$ is the subspace of $H^2(\Omega)$ consisting of
functions with vanishing value and normal derivative on
$\partial\Omega$.

The conforming finite element methods for the forth order problem
(\ref{wf}) require the finite element space be a subspace of
$H^2(\Omega)$. It is well known that constructing $H^2$ conforming
finite elements is generally quite challenging, specially in three
and higher dimensional spaces. Weak Galerkin finite element method,
first introduce in \cite{wy} (see also \cite{wy-mixed} and
\cite{mwy-wg-stabilization} for extensions), by design is to use
nonconforming elements to relax the difficulty in the construction
of conforming elements. Unlike the classical nonconforming finite
element method where standard derivatives are taken on each element,
the weak Galerkin finite element method relies on weak derives taken
as approximate distributions for the functions in nonconforming
finite element spaces. In general, weak Galerkin method refers to
finite element techniques for partial differential equations in
which differential operators (e.g., gradient, divergence, curl,
Laplacian) are approximated by weak forms as distributions.

A weak Galerkin method for the biharmonic equation has been derived
in \cite{mwy-bi} by using totally discontinuous functions of
piecewise polynomials on general partitions of arbitrary shape of
polygons/polyhedra. The key of the method lies in the use of a
discrete weak Laplacian plus a stabilization that is parameter-free.
In this paper, we will develop a new weak Galerkin method for the
biharmonic equation (\ref{pde})-(\ref{bc-n}) by redefining a weak
Laplacian, denoted by $\Delta_w$, for $C^0$ finite element
functions. Comparing with the WG method developed in \cite{mwy-bi},
the $C^0$-weak Galerkin finite element formulation has less number
of unknowns due to the continuity requirement. On the other hand,
due to the same continuity requirement, the $C^0$-WG method allows
only traditional finite element partitions (such as
triangles/quadrilaterals in 2D), instead of arbitrary
polygonal/polyhedral grids as allowed in \cite{mwy-bi}.

A suitably-designed interpolation operator is needed for the
convergence analysis of the $C^0$-weak Galerkin formulation. The
Scott-Zhang operator \cite{Scott-Zhang} turns out to serve the
purpose well with a refinement. This paper shall introduce a refined
version of the Scott-Zhang operator so that it preserves the volume
mass up to order $(k+1-d)$, and the surface mass up to order
$(k+2-d)$, when interpolating $H^1$ functions to the $P_{k+2}$
$C^0$-finite element space:
  \a{  Q_0 \ &  : \ H^1(\Omega) \to C^0\hbox{-}P_{k+2}, \\
       \int_T(v-Q_0 v) p dT &= 0  \quad \forall p\in P_{k+1-d}(T), \\
       \int_E(v-Q_0 v) p dE &= 0  \quad \forall p\in P_{k+2-d}(T), }
  where $T$ is any triangle ($d=2$) or tetrahedron ($d=3$) in
    the finite element, and $E$ is an edge or a face-triangle of $T$.
With the operator $Q_0$,  we can show an optimal order of approximation property of
   the $C^0$-finite element space, under the constraints
     of weak Galerkin formulation.
 Consequently, we show  optimal order of convergence in
   both a discrete $H^2$ norm and the $L^2$ norm, for the $C^0$
   weak Galerkin finite element solution.

The biharmonic equation models a plate bending problem, which is
   one of the first applicable problems of the finite element method,
   cf. \cite{Fraeijs,Argyris,Clough,Zlamal}.
The standard finite element method, i.e., the conforming element,
  requires a $C^1$ function space of piecewise polynomials.
This would lead to a high  polynomial degree
    \cite{Argyris,sz3,sz,Zenisek},
    or a macro-element \cite{Clough,ddps,Fraeijs,hhz,Powell,sz2},
    or a constraint element (where the polynomial degree is reduced
      at inter-element boundary) \cite{Bell,Percell,Zlamal}.
Mixed methods for the biharmonic equation avoid using $C^1$ element by
  reducing the fourth order equation to a system of two second
     order equations, \cite{ab,falk,gnp,monk,mwwy}.
Many other different  nonconforming and discontinuous finite element
   methods have been developed for solving the biharmonic equation.
Morley element \cite{morley} is a well known nonconforming element
    for its simplicity.
$C^0$ interior penalty methods are studied in \cite{bs, ghlmt, ms},
   which are similar to our $C^0$-weak Galerkin method except
  there is no penalty parameter here.

\section{Weak Laplacian and discrete weak Laplacian}
\label{Section:weak-Laplacian}

Let $D$ be a bounded polyhedral domain in $\mathbb{R}^d, d=2, 3$.
We use the standard definition for the Sobolev space $H^s(D)$
   and their associated inner products $(\cdot,\cdot)_{s,D}$,
   norms $\|\cdot\|_{s,D}$, and
   seminorms $|\cdot|_{s,D}$ for any $s\ge 0$.
When $D=\Omega$, we shall drop
   the subscript $D$ in the norm and in the inner product.

Let $T$ be a triangle or a tetrahedron with boundary
    $\partial T$.
A {\em weak function} on the region $T$ refers to a
  vector function $v=\{v_0, \bv_{n}\}$ such that $v_0\in L^2(T)$
    and $\bv_n\cdot\bn\in H^{-\frac12}(\partial T)$, where $\bn$
    is the outward normal direction of $T$ on its boundary.
The first component $v_0$  can be understood as the value of $v$ on $T$
   and the second component $\bv_n$ represents the value
           $\nabla v$ on the boundary of $T$.
Note that $\bv_n$ may not be necessarily related to the trace of
   $\nabla v_0$ on $\partial T$.
Denote by $\W(T)$ the space of all weak functions on $T$; i.e.,
\begin{equation}\label{www}
  \W(T) = \left\{v=\{v_0,\bv_n \}:\ v_0\in L^2(T),\;
       \bv_n\cdot\bn\in H^{-\frac12}(\partial T)\right\}.
\end{equation}
Let  $(\cdot,\cdot)_T$ stand for the $L^2$-inner product in
 $L^2(T)$, $\langle\cdot,\cdot\rangle_\pT$ be the inner product in
 $L^2(\pT)$. For convenience, define $G^2(T)$  as follows
$$
G^2(T)=\{\varphi: \ \varphi\in H^1(T),\  \Delta\varphi\in L^2(T)\}.
$$
It is clear that, for any $\varphi\in G^2(T)$, we have
 $\nabla\varphi\in H(\d,T)$.
It follows that
  $\nabla\varphi\cdot\bn\in  H^{-\frac12}(\partial T)$ for any
  $\varphi\in G^2(T)$.

\medskip

\begin{defi}
The dual of $L^2(T)$ can be identified with itself by using the
  standard $L^2$ inner product as the action of linear functionals.
With a similar interpretation, for any $v\in \W(T)$, the {\em weak
  Laplacian} of $v=\{v_0,\bv_n \}$ is defined as a linear
  functional $\Delta_w v$ in the dual space of $G^2(T)$ whose action
  on each $\varphi\in G^2(T)$ is given by
\begin{equation}\label{wl}
  (\Delta_w v, \ \varphi)_T = (v_0, \ \Delta\varphi)_T
     -\l v_0,\ \nabla\varphi\cdot\bn\r_\pT +\l
     \bv_n\cdot\bn, \ \varphi\r_\pT,
\end{equation}
   where $\bn$ is the outward normal direction to $\partial T$.
\end{defi}

The Sobolev space $H^2(T)$ can be embedded into the space $\W(T)$ by
  an inclusion map $i_\W: \ H^2(T)\to \W(T)$ defined as follows
$$
i_\W(\phi) = \{\phi|_{T},  \nabla\phi|_{\partial T}\},\qquad \phi\in H^2(T).
$$
With the help of the inclusion map $i_\W$, the Sobolev space $H^2(T)$
  can be viewed as a subspace of $\W(T)$ by identifying each $\phi\in
  H^2(T)$ with $i_\W(\phi)$.
Analogously, a weak function
  $v=\{v_0,\bv_n\}\in \W(T)$ is said to be in $H^2(T)$ if it can be
  identified with a function $\phi\in H^2(T)$ through the above
  inclusion map.
It is not hard to see that the weak Laplacian is
  identical with the strong Laplacian, i.e.,
  \a{ \Delta_w i_\W (v)=\Delta v } for
  smooth functions $v\in H^2(T)$.

Next, we introduce a discrete weak Laplacian operator by
approximating $\Delta_w$ in a polynomial subspace of the dual of
$G^2(T)$. To this end, for any non-negative integer $r\ge 0$, denote
by $P_{r}(T)$ the set of polynomials on $T$ with degree no more than
$r$. A discrete weak Laplacian operator, denoted by $\Delta_{w,r,T}$,
is defined as the unique polynomial $\Delta_{w,r,T}v \in P_r(T)$ that
satisfies the following equation
\begin{equation}\label{dwl}
  (\Delta_{w,r,T} v, \ \varphi)_T = ( v_0, \ \Delta\varphi)_T-\l v_0,\
  \nabla\varphi\cdot\bn\r_\pT +\l \bv_n\cdot\bn, \ \varphi\r_\pT \quad
  \forall \varphi\in P_r(T).
\end{equation}
Recall that $\bv_n$ represent the $\nabla v$ on $e\in\pT$.
Define $\bar{\bv}_n=(\nabla v\cdot\bn)\bn\equiv v_n\bn$.
Obviously, $\bar{\bv}_n\cdot\bn=\bv_n\cdot\bn$.
Since the quantity of interest is not $\bv_n$ but $\bv_n\cdot\bn$,
  we can replace $\bv_n$ by $\bar{\bv}_n=v_n\bn$ from now on
   to reduce the number of unknowns.
Scalar $v_n$ represents $\nabla\bv\cdot\bn$.

\section{Weak Galerkin Finite Element Methods}\label{Section:wg-fem}

Let ${\cal T}_h$ be a triangular ($d=2$) or a tetrahedral ($d=3$)
  partition of the domain $\Omega$
  with mesh size $h$.
Denote by ${\cal E}_h$ the set of all edges or  faces in ${\cal
T}_h$, and let ${\cal E}_h^0={\cal E}_h\backslash\partial\Omega$ be
the set of all interior edges or faces.

Since $\bv_n=v_n\bn$ with $v_n$ representing $\nabla v\cdot\bn$,
  obviously, $v_n$  is dependent on  $\bn$.
To ensure $v_n$ a single values function on $e\in\E_h$,
  we introduce a set of normal directions on ${\cal E}_h$ as follows
\begin{equation}\label{thanks.101}
{\cal D}_h = \{\bn_e: \mbox{ $\bn_e$ is unit and normal to $e$},\
e\in {\cal E}_h \}.
\end{equation}
Then, we can define a weak Galerkin finite element space
   $V_h$ for $k\ge 0$ as follows
\begin{equation}\label{Vh}
 V_h=\{v=\{v_0, v_{n}\bn_e\}:\ v_0\in V_0,\
   v_{n}|_e\in P_{k+1}(e), e\subset\partial T\},
  \end{equation}
where $v_n$ can be viewed as an approximation of $\nabla v\cdot\bn_e$ and
\begin{equation}\label{V0}
V_0=\{v\in H^1(\Omega); v|_T\in P_{k+2}(T)\}.
\end{equation}
Denote by $V_h^0$ a subspace
of $V_h$ with vanishing traces; i.e.,
$$
V_h^0=\{v=\{v_0,v_{n}\bn_e\}\in V_h, {v_0}|_e=0,\ {v_{n}}|_e=0,\
e\subset\partial T\cap\partial\Omega\}.
$$
Denote by $\Lambda_h$ the trace of $V_h$ on $\partial\Omega$ from
the component $v_0$. It is easy to see that $\Lambda_h$ consists of
piecewise polynomials of degree $k+2$. Similarly, denote by
$\Upsilon_h$ the trace of $V_h$ from the component $v_n$ as
piecewise polynomials of degree $k+1$.
Let  $\Delta_{w,k}$ be the discrete weak Laplacian operator on
the finite element space $V_h$ computed by using
(\ref{dwl}) on each element $T$ for $k\ge 0$; i.e.,
\an{\label{Dw}
  (\Delta_{w,k}v)|_T =\Delta_{w,k, T} (v|_T) \qquad \forall v\in
   V_h.
}
For simplicity of notation, from now on we shall drop the subscript
$k$ in the notation $\Delta_{w,k}$ for the discrete weak Laplacian.
We also introduce the following notation
$$
(\Delta_w v,\;\Delta_w w)_h=\sum_{T\in {\cal T}_h}(\Delta_w
v,\;\Delta_w w)_T.
$$
For any $u_h=\{u_0,u_{n}\bn_e\}$ and
  $v=\{v_0,v_{n}\bn_e\}$ in $V_h$, we introduce a bilinear form as
   follows
\an{\label{s}
  s(u_h, v)=\sum_{T\in\T_h} h_T^{-1}\langle \nabla u_0\cdot\bn_e-u_{n}, \
  \nabla v_0\cdot\bn_e-v_{n}\rangle_\pT.
  }
The stabilizer $s(u_h, v)$ defined above is to enforce a connection
   between the normal derivative of $u_0$ along $\bn_e$ and its approximation
   $u_n$.
\begin{algorithm}
A numerical approximation for (\ref{pde})-(\ref{bc-n}) can be
   obtained by seeking $u_h=\{u_0,\ u_{n}\bn_e\}\in V_h$
  satisfying $u_0=Q_b g$ and $u_{n}=(\bn\cdot\bn_e)Q_{n}\phi$
   on $\partial \Omega$ and the following equation:
\begin{equation}\label{wg}
(\Delta_w u_h,\ \Delta_w v)_h+s(u_h,\ v)=(f,\;v_0)  \quad\forall\
v=\{v_0,\ v_{n}\bn_e\}\in V_h^0,
\end{equation}
where $Q_b g$ and $Q_{n}\phi$ are the standard $L^2$ projections
onto the trace spaces $\Lambda_h$ and $\Upsilon_h$, respectively.
\end{algorithm}

\smallskip

\begin{lemma}
The weak Galerkin finite element scheme (\ref{wg}) has a unique
solution.
\end{lemma}

\begin{proof}
It suffices to show that the solution of (\ref{wg}) is trivial if
$f=g=\phi=0$. To this end, assume $f=g=\phi=0$ and take $v=u_h$  in
(\ref{wg}). It follows that
\[
(\Delta_w u_h,\Delta_w u_h)_h+s(u_h, u_h)=0,
\]
which implies that $\Delta_w u_h=0$ on each element $T$ and
$ \nabla u_0\cdot\bn_e=u_{n}$ on $\pT$. We claim that
$\Delta u_h=0$ holds true locally on each element $T$. To this end,
it follows from $\Delta_w u_h=0$ and (\ref{dwl}) that for any
$\varphi\in P_k(T)$ we have
\begin{eqnarray}
\label{april14.01}
  0=(\Delta_{w} u_h, \ \varphi)_T &=& (u_0, \ \Delta\varphi)_T-\l
  u_0,\ \nabla\varphi\cdot\bn\r_\pT +\l u_n\bn_e\cdot\bn, \
     \varphi\r_\pT \\
  &=&(\Delta u_0, \varphi)_T
   +\l u_n\bn_e\cdot\bn-\nabla u_0\cdot\bn,\ \varphi\r_\pT\nonumber\\
   &=&(\Delta u_0, \varphi)_T,\nonumber
\end{eqnarray}
where we have used
\begin{equation}\label{ne}
u_n\bn_e\cdot\bn-\nabla u_0\cdot\bn=\pm(u_n-\nabla u_0\cdot\bn_e)=0
\end{equation}
in the last equality. The identity (\ref{april14.01}) implies that
$\Delta u_0=0$ holds true locally on each element $T$. This,
together with  $\nabla u_0\cdot\bn_e=u_{n}$ on $\pT$,
shows that $u_h$ is a smooth harmonic function globally on $\Omega$.
The boundary condition of $u_0=0$ and $u_n=0$ then implies that $u_h\equiv 0$ on
$\Omega$, which completes the proof.
\end{proof}

\section{Projections: Definition and Approximation Properties}
\label{Section:L2projections}

In this section, we will introduce some locally defined
projection operators corresponding to the finite element space
$V_h$ with optimal convergent rates.

Let $Q_0 : H^1(\Omega) \to V_0$  be a special Scott-Zhang
    interpolation operator, to be defined in \eqref{s-z} in Appendix,
   such that for given
   $v\in H^1(\Omega) $, $Q_0v\in V_0$  and  for any $T\in\T_h$,
\begin{equation}\label{Q0}
  ( Q_0v, \ \Delta\varphi)_T-\l Q_0v,\
  \nabla\varphi\cdot\bn\r_\pT=
  ( v, \ \Delta\varphi)_T-\l v,\
  \nabla\varphi\cdot\bn\r_\pT, \quad\forall \varphi\in P_k(T),
  \end{equation}
and for $0\le s\le 2$
\begin{equation}\label{u-q0u}
   (\sum_{T\in\T_h}h^{2s}\|u-Q_0u\|^2_{s,T})^{1/2}\le C h^{k+3}\|u\|_{k+3}.
   \end{equation}
Now  we  can define an interpolation operator $Q_h$ from
    $ H^2(\Omega)$  to the finite element space $V_h$ such
    that on the element $T$, we have
\an{\label{Qh}
    Q_h u = \{Q_0 u, (Q_{n} (\nabla u\cdot\bn_e))\bn_e\},
   } where $Q_0$ is defined in \eqref{s-z} and $Q_n$
        is the $L^2$ projection onto $P_{k+1}(e)$, for each $e\subset
    \partial T$.
In addition, let $\bbQ_h$ be the local $L^2$ projection onto
   $P_{k}(T)$.
For any $\varphi\in P_k(T)$ we have
\begin{eqnarray*}
   (\Delta_{w} Q_h u,\ \varphi)_T &=&
     (Q_0 u,\ \Delta\varphi)_T-\langle Q_0 u,
      \ \nabla\varphi\cdot\bn \rangle_{\pT} + \langle
   Q_{n}(\nabla u\cdot\bn_e)\bn_e\cdot\bn, \;\varphi\rangle_{\pT}\\
&=&(u, \Delta\varphi)_T -\langle u,\
 \nabla\varphi\cdot\bn \rangle_{\pT} + \langle \nabla u\cdot\bn,
  \ \varphi\rangle_{\partial T}
  \\
  &=&(\Delta u,\ \varphi)_T=(\bbQ_h\Delta u,\ \varphi)_T,
  \end{eqnarray*}
which implies
\begin{equation}\label{key-Laplacian}
\Delta_{w} Q_h u = \bbQ_h (\Delta u).
\end{equation}
The above identity indicates that the discrete weak Laplacian of a
  projection of $u$ is a good approximation of the Laplacian of $u$.

Let $T\in {\cal T}_h$ be an element with $e$ as an edge or a face triangle.
It is well known that
   there exists a constant $C$ such that for any function $g\in
   H^1(T)$
\begin{equation}\label{trace}
  \|g\|_{e}^2 \leq C \left( h_T^{-1} \|g\|_T^2 + h_T \|\nabla
   g\|_{T}^2\right).
\end{equation}

Define a mesh-dependent semi-norm $\3bar\cdot\3bar$ in the finite element
space $V_h$ as follows
\begin{equation}\label{3barnorm}
\3bar v\3bar^2=(\Delta_wv,\ \Delta_w v)_h+s(v,\;v),\qquad v\in V_h.
\end{equation}
Using (\ref{trace}), we can derive the following
estimates which are useful in the convergence analysis for the
WG-FEM (\ref{wg}).

\begin{lemma}\label{l2}
Let $w\in H^{k+3}(\Omega)$ and $v\in V_h$. Then there exists a
constant $C$ such that the following estimates hold true.
\begin{eqnarray}
&  \sum_{T\in\T_h} |\langle \Delta w-\bbQ_h\Delta w,\; (\nabla
v_0-v_{n}\bn_e)\cdot\bn\rangle_\pT|
\leq C h^{k+1}\|w\|_{k+3} \3bar v\3bar,\label{mmm1}\\
&  \sum_{T\in\T_h} h_T^{-1}|\langle (\nabla
Q_0w)\cdot\bn_e-Q_{n}(\nabla w\cdot\bn_e),\nabla v_0\cdot\bn_e
-v_{n}\rangle_\pT|\label{mmm3}\\
&  \qquad \le Ch^{k+1}\|w\|_{k+3}\3bar v\3bar.\nonumber
\end{eqnarray}
\end{lemma}
\begin{proof}
To derive (\ref{mmm1}), we can use the Cauchy-Schwarz inequality, (\ref{ne}),  the
trace inequality (\ref{trace}), and the definition of $\bbQ_h$ to
obtain
\begin{eqnarray*}
&&\sum_{T\in\T_h} |\langle \Delta w-\bbQ_h\Delta w,\; (\nabla
v_0-v_{n}\bn_e)\cdot\bn\rangle_\pT|\\
&&\le
\left(\sum_{T\in\T_h} h_T\|\Delta w-\bbQ_h\Delta w\|_\pT^2\right)^{\frac12}\left(\sum_{T\in\T_h} h_T^{-1}
\|\nabla v_0\cdot\bn_e-v_{n}\|_\pT^2\right)^{\frac12}\\
&&\le C\left(\sum_{T\in\T_h}\left(\|\Delta w-\bbQ_h\Delta
w\|_T^2+h_T^2\|\nabla(\Delta w-\bbQ_h\Delta w)
\|_T^2\right)\right)^{\frac12}\3bar v\3bar\\
&&\le C h^{k+1}\|w\|_{k+3} \3bar v\3bar.
\end{eqnarray*}
As to (\ref{mmm3}), we have from the definition of $Q_{n}$, the
Cauchy-Schwarz inequality, the trace inequality (\ref{trace}),
and (\ref{u-q0u}) that
\begin{eqnarray*}
&&\sum_{T\in\T_h} h_T^{-1}|\langle (\nabla
Q_0w)\cdot\bn_e-Q_{n}(\nabla w\cdot\bn_e),\; \nabla v_0\cdot\bn_e
-v_{n}\rangle_\pT|\\
&&=\sum_{T\in\T_h} h_T^{-1}|\langle (\nabla Q_0w)\cdot\bn_e-\nabla
w\cdot\bn_e,\; \nabla v_0\cdot\bn_e -v_{n}\rangle_\pT|\\
&&\le \left(\sum_{T\in\T_h}h_T^{-1}\|(\nabla Q_0 w-\nabla w)\cdot\bn_e\|_\pT^2\right)^{\frac12}
\left(\sum_{T\in\T_h}h_T^{-1}\|\nabla v_0\cdot\bn_e-v_{n}\|^2_{\pT}\right)^{\frac12}\\
&& \le C\left(\sum_{T\in\T_h}(h_T^{-2}\|\nabla Q_0 w-\nabla w\|_T^2
+ \|\nabla Q_0 w-\nabla w\|_{1,T}^2)\right)^{\frac12}\3bar v\3bar\\
&&\le C h^{k+1}\|w\|_{k+3} \3bar v\3bar.
\end{eqnarray*}
 This completes the proof.
\end{proof}

\section{An Error Equation}

We first derive an equation that the projection of the exact
solution, $Q_hu$, shall satisfy. Using (\ref{dwl}), the
integration by parts, and (\ref{key-Laplacian}), we obtain
\begin{eqnarray*}
& & (\Delta_{w} Q_h u, \Delta_w v)_T \\
&=& (v_0, \Delta(\Delta_w Q_h
u))_T + \langle
(v_{n}\bn_e)\cdot\bn,\ \Delta_w Q_h u\rangle_{\pT}-\langle v_0, \nabla(\Delta_w Q_h u)\cdot\bn \rangle_{\pT}\nonumber\\
&=&(\Delta v_0,\ \Delta_w Q_h u)_T+\langle v_0, \nabla(\Delta_w
Q_h u)\cdot\bn\rangle_{\pT}-\langle\nabla v_0\cdot\bn, \Delta_w Q_h
u\rangle_{\pT}\nonumber\\
&& +\langle
(v_{n}\bn_e)\cdot\bn, \Delta_w Q_h u\rangle_{\pT}-\langle v_0, \nabla(\Delta_w Q_h u)\cdot\bn \rangle_{\pT}\nonumber\\
&=&(\Delta v_0,\Delta_w Q_h u)_T-\langle (\nabla v_0-v_{n}\bn_e)\cdot\bn,\Delta_w Q_h
u\rangle_{\pT}\nonumber\\
&=&(\Delta v_0,\bbQ_h\Delta u)_T-\langle (\nabla v_0-v_{n}\bn_e)\cdot\bn,
\bbQ_h\Delta u\rangle_{\pT}\nonumber\\
&=&(\Delta u,\Delta v_0)_T -\langle (\nabla v_0-v_{n}\bn_e)\cdot\bn,
\bbQ_h \Delta u\rangle_{\pT},
\end{eqnarray*}
which implies that
\begin{eqnarray}
(\Delta u,\Delta v_0)_T&=&(\Delta_{w} Q_h u, \Delta_w v)_T
+\langle (\nabla v_0-v_{n}\bn_e)\cdot\bn, \bbQ_h \Delta
u\rangle_{\pT}.\label{key}
\end{eqnarray}
Next, it follows from the integration by parts that
$$
(\Delta u,\Delta v_0)_T = (\Delta^2u, v_0)_T+\langle \Delta u,
\nabla v_0\cdot\bn\rangle_{\pT} - \langle\nabla(\Delta u)\cdot\bn,
v_0\rangle_{\pT}.
$$
Summing over all $T$ and then using the identity $(\Delta^2u,
v_0)=(f,v_0)$ we arrive at
\begin{eqnarray*}
\sum_{T\in\T_h}(\Delta u,\Delta v_0)_T& =&(f,v_0) +\sum_{T\in\T_h}\langle \Delta u,
\nabla v_0\cdot\bn\rangle_{\pT}\\
& =&(f,v_0) +\sum_{T\in\T_h}\langle \Delta u,
(\nabla v_0-v_{n}\bn_e)\cdot\bn\rangle_{\pT}.
\end{eqnarray*}
Combining the above equation with
(\ref{key}) leads to
\begin{eqnarray}\label{w-e}
(\Delta_w Q_h u, \Delta_w v)_h&=&(f,v_0)+\sum_{T\in\T_h}\langle
\Delta u-\bbQ_h\Delta u, (\nabla v_0-v_{n}\bn_e)\cdot\bn\rangle_{\pT}.
\end{eqnarray}
Define the error between the finite element approximation $u_h$ and the
projection of the exact solution $u$ as
$$
e_h:=\{e_0, \ e_{n}\bn_e\}=\{Q_0u-u_0,\
(Q_{n}(\nabla u\cdot\bn_e)-u_{n})\bn_e\}.
$$
Taking the difference of (\ref{w-e}) and (\ref{wg}) gives the
   following error equation
\begin{eqnarray}
  (\Delta_w e_h, \Delta_w v)_h+s(e_h,v)&=&\sum_{T\in\T_h}\langle
  \Delta u-\bbQ_h\Delta u,
  (\nabla v_0-v_{n}\bn_e)\cdot\bn\rangle_{\pT}\label{ee}  \\
   & +&s(Q_hu,v) \quad\quad\forall v\in V_h^0.\nonumber
\end{eqnarray}
Observe that the definition of
   the stabilization term $s(\cdot,\cdot)$ indicates that
\begin{eqnarray*}
s(Q_hu,v) &=&\sum_{T\in\T_h} h_T^{-1}\langle (\nabla
   Q_0u)\cdot\bn_e-Q_{n}(\nabla
   u\cdot\bn_e), \ \nabla v_0\cdot\bn_e-v_{n}\rangle_\pT.
\end{eqnarray*}

\section{Error Estimates}
First, we derive an estimate for the error function $e_h$ in the
natural triple-bar norm, which can be viewed as a discrete $H^2$-norm.

\begin{theorem} Let $u_h\in V_h$  be the weak Galerkin finite element
   solution arising from
   (\ref{wg}) with finite element functions of order $k+2\ge 2$. Assume
   that the exact solution of (\ref{pde})-(\ref{bc-n} ) is
   regular such that $u\in H^{k+3}(\Omega)$. Then, there exists a
   constant $C$ such that
\begin{equation}\label{err1}
   \3bar u_h-Q_hu\3bar \le Ch^{k+1}\|u\|_{k+3}.
\end{equation}
\end{theorem}

\begin{proof}
By letting $v=e_h$ in the error equation (\ref{ee}), we obtain the
  following identity
\begin{eqnarray*}
  \3bar e_h\3bar^2&=&\sum_{T\in\T_h}\langle
  \Delta u-\bbQ_h\Delta u,
  (\nabla e_0-e_{n}\bn_e)\cdot\bn\rangle_{\pT}\nonumber\\
  &+&\sum_{T\in\T_h} h_T^{-1}\langle (\nabla Q_0u\cdot\bn_e-Q_{n}(\nabla
  u\cdot\bn_e), \ \nabla e_0\cdot\bn_e-e_{n}\rangle_\pT.
\end{eqnarray*}
Using the estimates of Lemma \ref{l2}, we arrive at
\[
\3bar e_h\3bar^2 \le
Ch^{k+1}\|u\|_{k+3}\3bar
e_h\3bar,
\]
which implies (\ref{err1}). This completes the proof of the theorem.
\end{proof}

Next, we would like to provide an estimate for the standard
$L^2$ norm of the first component of the error function $e_h$.
Let us consider the following dual problem
\begin{eqnarray}
\Delta^2w&=& e_0\quad
\mbox{in}\;\Omega,\label{dual1}\\
w&=&0,\quad\mbox{on}\;\partial\Omega,\label{dual2}\\
\nabla w\cdot\bn&=&0\quad\mbox{on}\;\partial\Omega.\label{dual3}
\end{eqnarray}
The $H^{4}$ regularity assumption of the dual problem implies the
existence of a constant $C$ such that
\begin{equation}\label{reg}
\|w\|_4\le C\|e_0\|.
\end{equation}

\begin{theorem}
Let $u_h\in V_h$ be the weak Galerkin finite element solution
  arising from (\ref{wg}) with finite element functions of order
  $k+2\ge 3$. Assume that the exact solution of
  (\ref{pde})-(\ref{bc-n}) is  regular such that $u\in
  H^{k+3}(\Omega)$ and the dual problem (\ref{dual1})-(\ref{dual3}) has
  the $H^4$ regularity. Then, there exists a constant $C$ such that
\begin{equation}\label{err2}
  \|Q_0u-u_0\| \le Ch^{k+3}\|u\|_{k+3}.
\end{equation}
\end{theorem}
\begin{proof}
Testing (\ref{dual1}) by error function $e_0$  and then using
the integration by parts gives
\begin{eqnarray*}
\|e_0\|^2&=&(\Delta^2 w, e_0)\\
   &=&\sum_{T\in\T_h}(\Delta w,\Delta e_0)_T-
     \sum_{T\in\T_h}\l\Delta w,\nabla e_0\cdot\bn\r_\pT\\
   &=&\sum_{T\in\T_h}(\Delta w,\Delta e_0)_T
     -\sum_{T\in\T_h}\l\Delta w,(\nabla
    e_0-e_{n}\bn_e)\cdot\bn\r_\pT.
\end{eqnarray*}
 Using (\ref{key}) with $w$ in the place of
    $u$, we can rewrite the above equation as follows
\begin{eqnarray*}
   \|e_0\|^2&=&(\Delta_w Q_hw,\Delta_w e_h)_h
      -\sum_{T\in\T_h}\l\Delta w-\bbQ_h\Delta w,(\nabla
    e_0-e_{n}\bn_e)\cdot\bn\r_\pT.
\end{eqnarray*}
It now follows from the error equation (\ref{ee}) that
\begin{eqnarray*}
   ( \Delta_w Q_hw, \Delta_{w} e_h)_h&=&\sum_{T\in\T_h}\langle
   \Delta u-\bbQ_h\Delta u,
   (\nabla Q_0w-Q_{n}(\nabla w\cdot\bn_e)\bn_e)\cdot\bn\rangle_{\pT}\\
   &-&s(e_h,Q_hw)+s(Q_hu,Q_hw).
\end{eqnarray*}
Combining the two equations above gives
\begin{eqnarray}
    \|e_0\|^2&=&-\sum_{T\in\T_h}
    \langle\Delta w-\bbQ_h\Delta w,(\nabla e_0-e_{n}\bn_e)
     \cdot\bn\r_\pT \label{lll} \\
  & & + \sum_{T\in\T_h}\langle\Delta u-\bbQ_h\Delta u,
   (\nabla Q_0w-Q_{n}(\nabla w\cdot\bn_e)\bn_e)\cdot\bn\rangle_{\pT}
    \nonumber\\
  & & -s(e_h,Q_hw)+ s(Q_hu,Q_hw).\nonumber
\end{eqnarray}
Using the estimates of Lemma \ref{l2}, we can
  bound  two terms on the right-hand side of the equation
  above as follows
\begin{eqnarray*}
  \sum_{T\in\T_h}| \l\Delta w-\bbQ_h\Delta w,(\nabla e_0-e_{n}\bn_e)
         \cdot\bn\r_\pT | &\le&Ch^2\|w\|_4\3bar e_h\3bar,\\
        |s(e_h,Q_hw)|&\le&Ch^2\|w\|_4\3bar e_h\3bar.
\end{eqnarray*}
It follows from (\ref{ne}) and the definition of $Q_{n}$ and $Q_0$ that
\begin{eqnarray}
\|(\nabla Q_0w-Q_{n}(\nabla w\cdot\bn_e)\bn_e)\cdot\bn_e\|_\pT&=&\|(\nabla Q_0w-Q_{n}(\nabla w\cdot\bn_e)\bn_e)\cdot\bn\|_\pT\label{tt1}\\
=\|\nabla Q_0w\cdot\bn-Q_{n}(\nabla w\cdot\bn)\|_\pT
&\le& \|\nabla Q_0w\cdot\bn-\nabla w\cdot\bn\|_\pT\nonumber\\
+\|\nabla w\cdot\bn-Q_{n}(\nabla w\cdot\bn)\|_\pT
&\le&C\|\nabla Q_0w\cdot\bn-\nabla w\cdot\bn\|_\pT.\nonumber
\end{eqnarray}
Using (\ref{tt1}) and (\ref{trace}), we have
\begin{eqnarray*}
&&\sum_{T\in\T_h}\langle\Delta u-\bbQ_h\Delta u, (\nabla Q_0w-Q_{n}(\nabla w\cdot\bn_e)\bn_e)\cdot\bn\rangle_{\pT}\\
&&\le C\left(\sum_{T\in\T_h}h\|\Delta u-\bbQ_h\Delta u\|_\pT^2\right)^{1/2}\left(\sum_{T\in\T_h}h^{-1}\|(\nabla Q_0w-\nabla w)\cdot\bn\|^2_{\pT}\right)^{1/2}\\
&&\le C\left(\sum_{T\in\T_h}\left(\|\Delta u-\bbQ_h\Delta
u\|_T^2+h_T^2\|\nabla(\Delta u-\bbQ_h\Delta u)
\|_T^2\right)\right)^{\frac12}\cdot\\
&&\left(\sum_{T\in\T_h}\left(h^{-2}\|\nabla Q_0w- \nabla
w\|_T^2+\|\nabla(\nabla Q_0w- \nabla w)\|_T^2\right)\right)^{\frac12}\\
&&\le Ch^{k+3}\|u\|_{k+3}\|w\|_4.
\end{eqnarray*}
Using (\ref{tt1}) and  (\ref{trace}), we have
\begin{eqnarray*}
s(Q_hu, Q_hw)&=&\sum_{T\in\T_h} h^{-1}\langle \nabla ( Q_0u)\cdot\bn_e-Q_{n}(\nabla u\cdot\bn_e), \ \nabla (Q_0w)\cdot\bn_e-Q_{n}(\nabla w\cdot\bn_e)\rangle_\pT \\
&&\le \left(\sum_{T\in\T_h}h^{-1}\|\nabla Q_0u-\nabla u)\cdot\bn\|_\pT^2\right)^{1/2}\left(\sum_{T\in\T_h}h^{-1}\|\nabla Q_0w-\nabla w)\cdot\bn\|^2_{\pT}\right)^{1/2}\\
&\le&Ch^{k+3}\|u\|_{k+3}\|w\|_4.
\end{eqnarray*}
Substituting all above estimates into (\ref{lll}) and using (\ref{err1}) give
\[
\|e_0\|^2\le Ch^{k+3}\|u\|_{k+3}\|w\|_4.
\]
Combining the above estimate with (\ref{reg}), we obtain the desired $L^2$ error estimate (\ref{err2}).
\end{proof}

\section{Numerical Experiments}
This section shall report some numerical results for
   the $C^0$ weak Galerkin finite element methods for the following
   biharmonic equation:
\begin{eqnarray}
  \Delta^2 u&=&f \quad \mbox{ in } \Omega,\label{NumEx_problem}\\
  u&=&g\quad \mbox{ on } \partial\Omega,\label{NumEx_bd1}\\
  \frac{\partial u}{\partial n}&=&\psi\quad \mbox{ on }  \partial\Omega.
   \label{NumEx_bd2}
\end{eqnarray}
For simplicity, all the numerical experiments are conducted
   by using $k=0$ or $k=1$ in the finite element space $V_h$ in \eqref{Vh}.

If $\phi\in P_0(T)$ (i.e. $k=0$), the above equation can be simplified as
\begin{eqnarray*}
(\Delta_w v,\phi)_T=\langle v_n{\bf n}_e\cdot{\bf n},\ \phi\rangle_{\partial T}.
\end{eqnarray*}

The error for the $C^0$-WG solution will be measured in four norms
  defined as follows:
\a{
\intertext{$H^1$ semi-norm: }
  \|v-v_0\|_1^2
   &=\sum_{T\in\mathcal{T}_h}\int_T |\nabla v-\nabla v_0|^2dx.\\
\intertext{Discrete $H^2$ norm: }
  \3bar v\3bar^2
   &=\sum_{T\in\mathcal{T}_h}\|\Delta_w v\|^2_T
   +\sum_{T\in\mathcal{T}_h}h^{-1} \|\nabla v_0\cdot{\bf n}_e-v_n\|_\pT^2,\\
\intertext{Element-based $L^2$ norm: }
   \|Q_0v-v_0\|^2&=\sum_{T\in\mathcal{T}_h}\int_T|Q_0v-v_0|^2dx,\\
\intertext{Edge-based $L^2$ norm: }
  \|Q_n(\nabla v\cdot{\bf n}_e)-v_n\|_b^2
  &=\sum_{e\in\mathcal{E}_h}h\int_e |Q_n(\nabla v\cdot{\bf n}_e)-v_n|^2ds.
  }

\subsection{Example 1}
Consider the biharmonic problem (\ref{NumEx_problem})-(\ref{NumEx_bd2}) in the square domain $\Omega=(0,1)^2.$ Set the exact solution by
\begin{eqnarray*}
u=x^2(1-x)^2y^2(1-y)^2.
\end{eqnarray*}

\begin{table}[ht]
\caption{Example 1. Convergence rate for element $P_2(T)-P_1(e)$ $(k=0)$.}\label{ex1_tri}
\center
\begin{tabular}{||c||c|c|c|c||}
\hline\hline
$h$ & $\|u-u_0\|_1$ &$\3bar u_h-Q_h u\3bar$ & $\|u_0-Q_0 u\|$ & $\|Q_n(\nabla u\cdot{\bf n}_e)-u_n\|_b$ \\
\hline\hline
   1/4    &6.8858e-03   &6.0250e-02   &1.4563e-03   &4.3364e-03\\ \hline
   1/8    &1.7465e-03   &3.0867e-02   &3.8153e-04   &1.4617e-03\\ \hline
   1/16   &4.3885e-04   &1.5555e-02   &9.6991e-05   &4.0941e-04\\ \hline
   1/32   &1.0982e-04   &7.7916e-03   &2.4350e-05   &1.0558e-04\\ \hline
   1/64   &2.7458e-05   &3.8972e-03   &6.0931e-06   &2.6601e-05\\ \hline
   1/128  &6.8645e-06   &1.9487e-03   &1.5236e-06   &6.6629e-06\\ \hline
 \mbox{Conv.Rate}&1.9949  &9.9160e-01  &1.9829   &1.8865 \\ \hline
\hline
\end{tabular}
\end{table}

It is easy to check that
$$u|_{\partial\Omega}=0,\  \frac{\partial u}{\partial n}=0.$$
The function $f$ is given according to the equation (\ref{NumEx_problem}).

The test is performed by using uniform triangular mesh.
The mesh is constructed as follows:
    1) partition the domain into $n\times n$ sub-rectangles;
   2) divide each square element into two triangles by the diagonal
     line with a negative slope.
The mesh size is denoted by $h=1/n.$ Table \ref{ex1_tri}
   shows the convergence rate for $C^0$-WG solutions based on
    $k=0$ in four norms respectively.
The numerical results indicate that the WG solution is convergent with
    rate $O(h^2)$ in $H^1$, $O(h^1)$ in $H^2$, and $O(h^2)$ in $L^2$ norms.
The convergence rate for $\|Q_n(\nabla u\cdot{\bf n}_e)-u_n\|_b$
   is $O(h^2)$.
Also, the same problem is tested for $k=1$.
   The results are reported in Table \ref{ex1_tri1}.
It indicates that the WG solution is convergent with rate
   $O(h^3)$ in $H^1$, $O(h^2)$ in $H^2$, and $O(h^4)$ in $L^2$ norms.
We note that the $L^2$ error is convergent at order $4$, two orders higher
   than that of $k=0$, confirming the sharpness of Theorem \ref{err2}.
 Moreover the convergence rate for $\|Q_n(\nabla u\cdot{\bf n}_e)-u_n\|_b$
    is $O(h^3)$, for $k=1$.

\begin{table}[ht]
\caption{Example 1. Convergence rate element $P_3(T)-P_2(e)$ $(k=1)$}\label{ex1_tri1}
\center
\begin{tabular}{||c||c|c|c|c||}
\hline\hline
$h$ & $\|u-u_0\|_1$ &$\3bar u_h-Q_h u\3bar$ & $\|u_0-Q_0 u\|$
   & $\|Q_n(\nabla u\cdot{\bf n}_e)-u_n\|_b$ \\
\hline\hline
   1/4    &1.5888e-03   &1.5888e-02   &1.5751e-04   &1.7898e-03\\ \hline
   1/8    &2.6787e-04   &4.7921e-03   &1.3887e-05   &2.6200e-04\\ \hline
   1/16   &3.8354e-05   &1.2963e-03   &1.0006e-06   &3.4742e-05\\ \hline
   1/32   &5.0893e-06   &3.3568e-04   &6.6590e-08   &4.4563e-06\\ \hline
   1/64   &6.5373e-07   &8.5314e-05   &4.2842e-09   &5.6344e-07\\ \hline
   1/128  &8.2783e-08   &2.1499e-05   &2.7341e-10   &7.0798e-08\\ \hline
 \mbox{Conv.Rate}&2.8597  &1.9152  &3.8450   &2.9336 \\ \hline
\hline
\end{tabular}
\end{table}

\subsection{Example 2}
In this problem, we set $\Omega=(0,1)^2$ and the exact solution:
$$u=\sin(\pi x)\sin(\pi y),$$
with
$$u|_{\partial\Omega}=0, \ \frac{\partial u}{\partial n}\neq 0.$$
Boundary conditions and $f$ are given according to the equation
 (\ref{NumEx_problem})-(\ref{NumEx_bd2}).

Again, the uniform triangular mesh is used in the experiment.
Table \ref{ex2_tri} shows that the convergence rate
  for $C^0$-WG solutions in $H^1$, $H^2$ and $L^2$ norms is $O(h^2)$,
    $O(h)$ and $O(h^2)$, respectively.

\begin{table}[ht]
\caption{Example 2. Convergence rate for element $P_2(T)-P_1(e)$ $(k=0)$.}\label{ex2_tri}
\center
\begin{tabular}{||c||c|c|c|c||}
\hline\hline
$h$ & $\|u-u_0\|_1$ &$\3bar u_h-Q_h u\3bar$ & $\|u_0-Q_0 u\|$
  & $\|Q_n(\nabla u\cdot{\bf n}_e)-u_n\|_b$ \\
\hline\hline
   1/4     &6.1653e-01   &5.5381       &1.2978e-01   &2.7515e-01\\ \hline
   1/8     &1.4737e-01   &2.7431       &3.2219e-02   &6.8563e-02\\ \hline
   1/16    &3.6122e-02   &1.3640       &7.9854e-03   &1.6489e-02\\ \hline
   1/32    &8.9758e-03   &6.8082e-01   &1.9899e-03   &4.0589e-03\\ \hline
   1/64    &2.2403e-03   &3.4024e-01   &4.9703e-04   &1.0102e-03\\ \hline
   1/128   &5.5983e-04   &1.7010e-01   &1.2423e-04   &2.5224e-04\\ \hline
 \mbox{Conv.Rate}&2.0186   &1.0046   &2.0058   &2.0209\\ \hline
\hline
\end{tabular}
\end{table}

\subsection{Example 3}
The exact solution is chosen as
$$u=\sin(\pi x)\cos(\pi y),$$
with nonhomogeneous boundary conditions.

Table \ref{ex3_tri} shows that the convergence rate for
 $C^0$-WG solutions in $H^1$, $H^2$ and $L^2$ norms is $O(h^2)$,
  $O(h)$, and $O(h^2)$, respectively.

\begin{table}[ht]
\caption{Example 3. Convergence rate for element $P_2(T)-P_1(e)$ $(k=0)$.}\label{ex3_tri}
\center
\begin{tabular}{||c||c|c|c|c||}
\hline\hline
$h$ & $\|u-u_0\|_1$ &$\3bar u_h-Q_h u\3bar$ & $\|u_0-Q_0 u\|$
   & $\|Q_n(\nabla u\cdot{\bf n}_e)-u_n\|_b$ \\
\hline\hline
   1/4     &2.7134e-01   &4.3389       &2.8817e-02   &5.9389e-01\\ \hline
   1/8     &5.6175e-02   &2.4888       &5.8917e-03   &2.0490e-01\\ \hline
   1/16    &1.3236e-02   &1.3196       &1.3285e-03   &5.9347e-02\\ \hline
   1/32    &3.2856e-03   &6.7374e-01   &3.2089e-04   &1.5585e-02\\ \hline
   1/64    &8.2159e-04   &3.3917e-01   &7.9441e-05   &3.9554e-03\\ \hline
   1/128   &2.0553e-04   &1.6994e-01   &1.9812e-05   &9.9329e-04\\ \hline
 \mbox{Conv.Rate}&2.0608   &9.4191e-01   &2.0916   &1.8609\\ \hline
\hline
\end{tabular}
\end{table}

\subsection{Example 4}
In the final example, we test the a case where the exact solution
has a low regularity in the domain $\Omega=(0,1)^2$. The exact
solution is given by
$$
u=r^{3/2}\bigg(\sin\frac{3\theta}{2}-3\sin\frac{\theta}{2}\bigg),
$$
where $(r,\theta)$ are the polar coordinates. It is known that $u\in
H^{2.5}(\Omega)$. The performance for $C^0$ weak Galerkin finite
element approximations for element $P_2(T)-P_1(e)$ $(k=0)$ is
reported in Table \ref{ex4_tri}. The convergence rates in
$H^1$-norm, $H^2-$norm, and edge-based $L^2$-norm are seen as
$O(h^{1.4})$, $O(h^{0.47})$, and $O(h^{1.4})$. The corresponding
theoretical prediction has the order of $O(h^{1.5})$, $O(h^{0.5})$,
and $O(h^{1.5})$. We believe that the numerical results are in
consistency with the theory. Table \ref{ex4_tri} indicates that the
numerical convergence rate in the standard $L^2$ is of order
$O(h^{1.88})$, which exceeds the theoretical prediction of
$O(h^{1.5})$.

Table \ref{ex4_tri0} contains some numerical results for the weak
Galerkin element $P_3(T)-P_2(e)$ $(k=1)$. The convergence rates in
$H^1$-norm, $H^2-$norm, and edge-based $L^2$-norm are seen as
$O(h^{1.5})$, $O(h^{0.5})$, and $O(h^{1.5})$, which are completely
in consistency with the theory. For the element-based $L^2$ error,
Table \ref{ex4_tri0} indicates a numerical convergence rate of order
$O(h^{2.49})$, which is also consistent with the theoretical
prediction of $O(h^{2.5})$.

\begin{table}[ht]
\caption{Example 4. Convergence rate for element $P_2(T)-P_1(e)$ $(k=0)$.}\label{ex4_tri}
\center
\begin{tabular}{||c||c|c|c|c||}
\hline\hline
$h$ & $\|u-u_0\|_1$ &$\3bar u_h-Q_h u\3bar$ & $\|u_0-Q_0 u\|$
   & $\|Q_n(\nabla u\cdot{\bf n}_e)-u_n\|_b$ \\
\hline\hline
   1/4     &3.1965e-02   &9.0667e-01   &3.3386e-03   &1.5615e-01\\ \hline
   1/8     &1.3596e-02   &6.8589e-01   &1.1209e-03   &6.2562e-02\\ \hline
   1/16    &5.1368e-03   &4.9952e-01   &3.1392e-04   &2.3370e-02\\ \hline
   1/32    &1.8697e-03   &3.5808e-01   &8.2158e-05   &8.4733e-03\\ \hline
   1/64    &6.7020e-04   &2.5488e-01   &2.0925e-05   &3.0321e-03\\ \hline
   1/128   &2.3855e-04   &1.8081e-01   &5.2718e-06   &1.0784e-03\\ \hline
 \mbox{Conv.Rate}&1.4233   &4.6844e-01   &1.8767   &1.4415\\ \hline
\hline
\end{tabular}
\end{table}

\begin{table}[ht]
\caption{Example 4. Convergence rate for element $P_3(T)-P_2(e)$ $(k=1)$.}\label{ex4_tri0}
\center
\begin{tabular}{||c||c|c|c|c||}
\hline\hline
$h$ & $\|u-u_0\|_1$ &$\3bar u_h-Q_h u\3bar$ & $\|u_0-Q_0 u\|$
   & $\|Q_n(\nabla u\cdot{\bf n}_e)-u_n\|_b$ \\
\hline\hline
   1/4     &2.5197e-02   &5.0303e-01   &1.3671e-03   &4.7712e-02\\ \hline
   1/8     &8.9650e-03   &3.5619e-01   &2.4629e-04   &1.6900e-02\\ \hline
   1/16    &3.1718e-03   &2.5190e-01   &4.3679e-05   &5.9764e-03\\ \hline
   1/32    &1.1215e-03   &1.7812e-01   &7.7825e-06   &2.1130e-03\\ \hline
   1/64    &3.9652e-04   &1.2595e-01   &1.3812e-06   &7.4708e-04\\ \hline
   1/128   &1.4019e-04   &8.9063e-02   &2.4431e-07   &2.6413e-04\\ \hline
 \mbox{Conv.Rate}&   1.4984 &  4.9966e-01 &  2.4907 &  1.4995\\ \hline
\hline
\end{tabular}
\end{table}
\appendix

\section{A mass-preserving Scott-Zhang operator}
We will prove the existence of an interpolation $Q_0$ used
    in \eqref{Q0} and in the previous section,
    which is a special Scott-Zhang operator\cite{Scott-Zhang}.
The new Scott-Zhang operator preserves the mass on each element
   and on each face, of
    four orders and three orders less, respectively, when interpolating
    $H^1(\Omega)$ functions to the finite element $V_h$ functions.
We shall derive the optimal-order approximation properties for the
  interpolation in the section, which leads to a quasi-optimal convergence
   of the weak Galerkin finite element method (\ref{wg}).

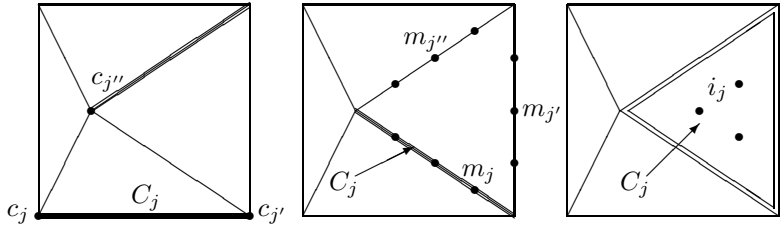
\begin{figure}[htb]\setlength\unitlength{1pt}\begin{center}
    \begin{picture}(280,85)(0,0)
  \def\gd{   \put(0,0){\line(1,0){80}} \put(0,80){\line(1,0){80}}
   \put(0,0){\line(0,1){80}} \put(80,0){\line(0,1){80}}
   \put(20,40){\line(3,-2){60}} \put(20,40){\line(3,2){60}}
   \put(20,40){\line(-1,-2){20}} \put(20,40){\line(-1,2){20}} }
   \put(0,0){\begin{picture}(80,80)(0,0)\gd
    \put(20,40){\circle*{3}}
        \multiput(20,40.8)(0,-1.6){2}{\line(3,2){60}}
        \multiput(0,0)(80,0){2}{\circle*{3}}
        \multiput(0,0.6)(0,-1.2){2}{\line(1,0){80}}
     \put(20,50){$c_{j''}$} \put(-12,0){$c_j$}\put(83,0){$c_{j'}$}
      \put(35,5){$C_j$}
      \end{picture}}
   \put(100,0){\begin{picture}(80,80)(0,0) \gd
    \multiput(35,50)(15,10){3}{\circle*{3}}
        \multiput(20,40.8)(0,-1.6){2}{\line(3,-2){60}}
        \multiput(35,30)(15,-10){3}{\circle*{3}}
        \multiput(80,20)(0,20){3}{\circle*{3}}
    \put(38,65){$m_{j''}$} \put(60,15){$m_j$} \put(83,38){$m_{j'}$}
      \put(10,10){$C_j$} \put(20,15){\vector(2,1){22}}
      \end{picture}}
   \put(200,0){\begin{picture}(80,80)(0,0) \gd
     \put(50,40){\circle*{3}}\put(65,50){\circle*{3}}\put(65,30){\circle*{3}}
    \put(53,46){$i_j$}
    \put(23,40){\line(3,2){55}}\put(23,40){\line(3,-2){55}}
    \put(78,3){\line(0,1){74}}
      \put(20,10){$C_j$} \put(30,15){\vector(1,1){20}}
      \end{picture}}

    \end{picture}\end{center}
 \caption{\label{node}
   Averaging patches ($C_j$) in 2D (an edge or a triangle),
    for corner nodes ($c_j$), middle nodes ($m_j$)
    and internal nodes ($i_j$), in 2D. }
\end{figure}

\long\def\comment#1{}
\comment {

 }

The original Scott-Zhang operator maps $u\in H^1(\Omega)$ functions to
   $C^0$-Lagrange finite element functions, preserving the zero boundary
   condition if $u\in H^1(\Omega)$.
It is an Lagrange interpolation.
All the Lagrange nodes (\cite{Brenner-Scott})
     on one element are classified into three types:
   \a{& \hbox{corner nodes} \  c_j:  && \
             \hbox{3 vertex nodes in 2D, or all edge nodes in 3D}, \\
      & \hbox{middle nodes} \     m_j:  & & \
          \hbox{all mid-edge nodes in 2D, or mid-triangle nodes in 3D}, \\
     & \hbox{internal nodes} \   i_j:   && \
           \hbox{all internal nodes in the triangle/tetrahedra. } }
The three types of nodes are illustrated in Figures \ref{node} and \ref{node3d}.
In simple words,  $\{c_j\}$ are nodes shared by possibly more than two elements,
   $\{m_j\}$ are nodes shared by no more than two elements,
   and $\{i_j\}$ are nodes internal to one element.

\begin{figure}[htb]\setlength\unitlength{0.9pt}\begin{center}
    \begin{picture}(320,120)(0,0)
  \def\gd{ \put(0,0){\line(1,0){80}} \put(0,80){\line(1,0){80}}
   \put(0,0){\line(0,1){80}} \put(80,0){\line(0,1){80}}
   \put(0,80){\line(1,-1){80}} \put(80,80){\line(-3,1){60}}
   \put(80,0){\line(1,1){20}} \put(100,100){\line(-1,-5){20}}
   \put(100,100){\line(-1,0){80}} \put(100,100){\line(-1,-1){20}}
   \put(100,100){\line(0,-1){80}}\put(20,100){\line(-1,-1){20}}
    \def\la{\vrule width.4pt height.4pt}
 \multiput(  20.00,  20.00)(  -2.357,  -2.357){  8}{\la}
 \multiput(  20.00,  20.00)(   3.162,  -1.054){ 18}{\la}
 \multiput(  20.00,  20.00)(   3.333,   0.000){ 24}{\la}
 \multiput(  20.00,  20.00)(   0.000,   3.333){ 24}{\la}
 \multiput(   0.00,   0.00)(   0.654,   3.269){ 30}{\la}
 \multiput(  20.00, 100.00)(   1.715,  -2.858){ 34}{\la}
 \multiput(  20.00, 100.00)(   2.357,  -2.357){ 33}{\la}
    }
   \put(0,0){\begin{picture}(80,80)(0,0)\gd
    \multiput(0,0)(20,0){5}{\circle*{3}}
     \multiput(0,0)(0,20){5}{\circle*{3}}
     \multiput(80,0)(-20,20){5}{\circle*{3}}
    \put(-12,81){$c_i$} \put(18,5.5){$c_j$}
     \put(75,2){\line(-1,1){73}}\put(2,2){\line(1,0){73}}
      \put(2,2){\line(0,1){73}}
     \put(0,103){$C_j$}\put(6,98){\vector(0,-1){26.5}}
      \end{picture}}
   \put(110,0){\begin{picture}(80,80)(0,0) \gd
     \put(22,21){\line(3,-1){54}}\put(22,21){\line(0,1){73}}
      \put(76,3){\line(-3,5){54}}
     \put(35,33){\circle*{3}}\put(35,50){\circle*{3}}
     \put(50,26){\circle*{3}} \put(22,25){$m_j$}
     \put(25,107){$C_j$}\put(28,105){\vector(0,-1){20}}
      \end{picture}}
   \put(220,0){\begin{picture}(80,80)(0,0) \gd
     \put(30,76){\circle*{3}}\put(28,65){$i_j$}
    \put(78.5,3){\line(0,1){76}}\put(78,4){\line(-1,1){76}}
    \put(78.5,5){\line(-3,5){56}}
    \put(78.5,78.5){\line(-3,1){56}}
    \put(2,79){\line(1,0){77}} \put(2,79){\line(1,1){19}}
    \put(45,105){$C_j$}\put(48,103){\vector(0,-1){14}}
      \end{picture}}

    \end{picture}\end{center}

\caption{\label{node3d} Averaging patches  ($C_j$) in 3D
       (a triangle or a tetrahedron),
    for corner nodes ($c_j$), middle nodes ($m_j$)
    and internal nodes ($i_j$), in 3D. }
\end{figure}
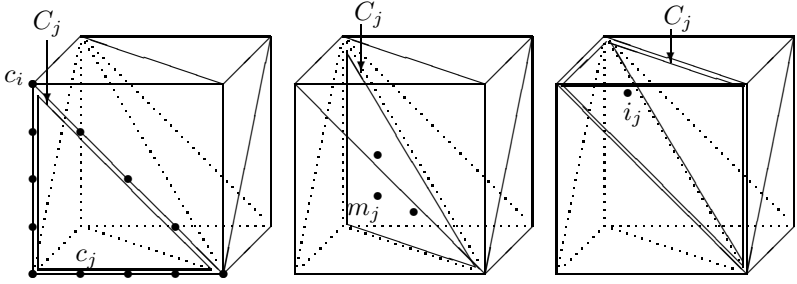

A Lagrange nodal basis function $\phi_j$
    is a $P_{k+2}$ polynomial which assumes value 1 at one node $c_j$, but
    vanishes at all other $\dim P_{k+2}^d-1$ nodes.
For example, a $P_4^2$ nodal basis function on the reference triangle
   $\{ 0 \le x, y, 1-x-y \le 1 \}$, at node $(1/4,0)$, c.f. Figure
     \ref{phi-2-g}, is
   \an{\label{phi-2} \phi_2(x,y) = \frac{x  (1-x-y)(3/4-x-y)(2/4-x-y)}
                   {(1/4) (1-1/4-0)(3/4-1/4-0)(2/4-1/4-0)}. }
The restriction of a nodal basis $\phi_j$ on a lower dimensional
  simplex, a triangle or an edge or a vertex, is also a nodal basis function
    on that lower dimensional finite element.
For example, this node basis function \eqref{phi-2}
     is the restriction of the
   following 3D nodal basis function (at node $(1/4,0,0)$ on tetrahedron
    $\{ 0 \le x, y, z, 1-x-y-z \le 1 \}$) on the reference triangle,
   \an{\label{phi-3d}
     \phi_j(x,y,z)  = \frac{x  (1-x-y-z)(3/4-x-y-z)(2/4-x-y-z)}
                   {(1/4) (1-1/4-0-0)(3/4-1/4-0-0)(2/4-1/4-0-0)}. }
The restriction of 2D basis function $\phi_2$ in \eqref{phi-2} in 1D is,
    c.f. Figure \ref{phi-2-g},
   \an{\label{phi-1d}
     \phi_{j'}(x)  = \frac{x  (1-x )(3/4-x )(2/4-x )}
                   {(1/4) (1-1/4 )(3/4-1/4 )(2/4-1/4 )}. }

\begin{figure}[htb]\setlength\unitlength{1.5pt}\begin{center}
    \begin{picture}(200,80)(0,0)
    \put(0,0){\setlength\unitlength{1.3pt}\begin{picture}(80,80)(0,0)
    \multiput(0,0)(20,0){5}{\circle*{3}}
    \multiput(20,20)(20,0){2}{\circle*{3}}\multiput(20,40)(20,0){1}{\circle*{3}}
     \multiput(0,0)(0,20){5}{\circle*{3}} \put(20,0){\circle{5}}
     \multiput(80,0)(-20,20){5}{\circle*{3}}
     \put(18,5.5){$\phi_2$}
     \put(80,0){\line(-1,1){80}}\put(0,0){\line(1,0){80}}
      \put(0,0){\line(0,1){80}}
      \end{picture}}
    \put(120,0){\begin{picture}(80,80)(0,0)
      \put(0,0){\line(1,0){80}}\multiput(0,0)(20,0){5}{\circle*{3}}
      \put(20,0){\circle{5}} \put(17,6){$\phi_{j'}$ (in 1D)}
  \put(   0.00,  20.00){\circle{2}} \put(   0.00,  20.00){\line(1,0){80}}
  \put(20.00,  70.00){\circle{2}}
  \put(40.00,  20.00){\circle{2}}
  \put(60.00,  20){\circle{2}}
  \put(80.00,  20.00){\circle{2}}
     \def\la{\vrule width.4pt height.4pt}
 \multiput(   0.00,  20.00)(   0.038,   0.331){ 65}{\la}
 \multiput(   2.50,  41.76)(   0.052,   0.329){ 48}{\la}
 \multiput(   5.00,  57.60)(   0.076,   0.325){ 33}{\la}
 \multiput(   7.50,  68.32)(   0.122,   0.310){ 20}{\la}
 \multiput(  10.00,  74.69)(   0.226,   0.245){ 11}{\la}
 \multiput(  12.50,  77.40)(   0.331,  -0.036){  7}{\la}
 \multiput(  15.00,  77.13)(   0.229,  -0.243){ 10}{\la}
 \multiput(  17.50,  74.47)(   0.163,  -0.291){ 15}{\la}
 \multiput(  20.00,  70.00)(   0.132,  -0.306){ 18}{\la}
 \multiput(  22.50,  64.22)(   0.118,  -0.312){ 21}{\la}
 \multiput(  25.00,  57.60)(   0.111,  -0.314){ 22}{\la}
 \multiput(  27.50,  50.55)(   0.111,  -0.314){ 22}{\la}
 \multiput(  30.00,  43.44)(   0.114,  -0.313){ 21}{\la}
 \multiput(  32.50,  36.58)(   0.122,  -0.310){ 20}{\la}
 \multiput(  35.00,  30.25)(   0.136,  -0.304){ 18}{\la}
 \multiput(  37.50,  24.67)(   0.157,  -0.294){ 15}{\la}
 \multiput(  40.00,  20.00)(   0.189,  -0.275){ 13}{\la}
 \multiput(  42.50,  16.37)(   0.235,  -0.237){ 10}{\la}
 \multiput(  45.00,  13.85)(   0.292,  -0.162){  8}{\la}
 \multiput(  47.50,  12.46)(   0.331,  -0.036){  7}{\la}
 \multiput(  50.00,  12.19)(   0.319,   0.097){  7}{\la}
 \multiput(  52.50,  12.95)(   0.277,   0.186){  9}{\la}
 \multiput(  55.00,  14.63)(   0.239,   0.232){ 10}{\la}
 \multiput(  57.50,  17.05)(   0.216,   0.254){ 11}{\la}
 \multiput(  60.00,  20.00)(   0.205,   0.263){ 12}{\la}
 \multiput(  62.50,  23.20)(   0.207,   0.261){ 12}{\la}
 \multiput(  65.00,  26.35)(   0.226,   0.245){ 11}{\la}
 \multiput(  67.50,  29.06)(   0.267,   0.200){  9}{\la}
 \multiput(  70.00,  30.94)(   0.325,   0.074){  7}{\la}
 \multiput(  72.50,  31.51)(   0.298,  -0.149){  8}{\la}
 \multiput(  75.00,  30.25)(   0.189,  -0.275){ 13}{\la}
 \multiput(  77.50,  26.62)(   0.118,  -0.312){ 21}{\la}
      \end{picture}}

    \end{picture}\end{center}

\caption{
    A 2D nodal basis $\phi_2$ \eqref{phi-2}
    and its restriction in  1D, $\phi_{j'}$  \eqref{phi-1d}.
   \label{phi-2-g} }
\end{figure}
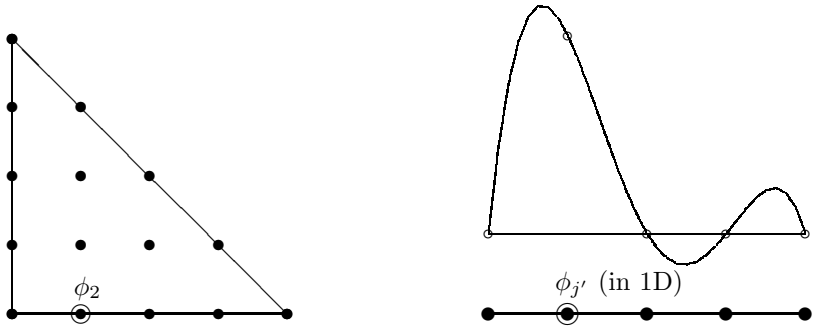

\comment{


t=0:2.5:80; t=t'; x=t/80; y=1/4;
g=[t*4 x.*(1-x).*(3/4-x).*(2/4-x)./(y.*(1-y).*(3/4-y).*(2/4-y))];
plot(x,g(:,2))
g(:,2)=g(:,2)*12+40;
[0*t(1:end-1)+99 g(1:end-1,:) g(2:end,:)]

tp=-485/128 +2865/32* x -21105/64*x.^2 +13615/32*x.^3 -11655/64*x.^4
plot(x,tp)
g=[t*4 tp*12+40];
[0*t(1:end-1)+99 g(1:end-1,:) g(2:end,:)]

p4:=a0 +a10*x +  a20*x^2 + a30*x^3+ a40*x^4  :
ph2:=(x,y) -> x*(1-x-y)*(3/4-x-y)*(2/4-x-y):
phi2:=ph2(x,0)/ph2(1/4,0):
i2:=int(phi2*phi2,x=0..1):
  e1:=int(p4*phi2,x=0..1)=1:

f:=1: g:=int(f*phi2/i2, x=0..1):
  e2:=int(p4*(f-g*phi2),x=0..1)=0:
f:=x: g:=int(f*phi2/i2, x=0..1):
  e3:=int(p4*(f-g*phi2),x=0..1)=0:
f:=x^2: g:=int(f*phi2/i2, x=0..1):
  e4:=int(p4*(f-g*phi2),x=0..1)=0:
f:=x^3: g:=int(f*phi2/i2, x=0..1):
  e5:=int(p4*(f-g*phi2),x=0..1)=0:
f:=x^4: g:=int(f*phi2/i2, x=0..1):
  e6:=int(p4*(f-g*phi2),x=0..1)=0:
s:=solve({e1,e2,  e4,e5,e6},
    {a0, a10,a20, a30, a40 });
subs(s,p4);
}

On each element $T$ (an edge, a triangle, or a tetrahedron),  the $P_k$
   Lagrange basis $\{\phi_j\}$ has a dual basis $\{\psi_j\in P_k^d \}$,
   satisfying
   \an{\label{dual} \int_T \phi_j \psi_{j'} dx = \delta_{jj'}=\begin{cases}
            1 & \hbox{ if } j=j', \\
        0 & \hbox{ if } j\ne j'. \end{cases} }
In other words, if writing $\{ \psi_j\}$ as linear combinations of Lagrange
    basis $\{\phi_j\}$, the coefficients are simply the inverse matrix
    of the mass matrix, the $L^2$-matrix of $\{\phi_j\}$.
For example, the dual basis function $\psi_2$ for the nodal basis function
   $\phi_2$ in \eqref{phi-2} (2D) is
  \a{
      \psi_2^{[2D]}(x,y) &= \frac{2835}4 x - \frac{12285}4x^2 +
                   \frac{8505}2 x^3 +  8505 x^2 y
                 +\frac{8505}2 xy^2 \\
              &\qquad - 1890 x^4  - 5670 x^3  y - 5670 x^2  y^2  - 1890 x y^3.
    }
We can compute the dual of $\psi_{j'}$ in \eqref{phi-1d} in 1D to get
   \an{\label{psi-1}
   \psi_{j'}^{[1D]}(x) &=-\frac{485}{128} +\frac{2865}{32} x
          -\frac{21105}{64}x^2 +\frac{13615}{32}x^3 -\frac{11655}{64}x^4\\
    &=-\frac{485}{128}\phi_0+ \frac{64225}{16384}\phi_1
          +\frac{345}{1024}\phi_2 -\frac{4255}{16384}\phi_3
          -\frac{85}{128}\phi_4,
    \nonumber }
   where $\phi_i$ is the nodal basis on $[0,1]$
        at $x_i=i/4$, $i=0,1,2,3,4. $
The dual function $\psi_{j'}^{[1D]}(x)$ in \eqref{psi-1} is plotted in
    Figure \ref{psi-1-g}.
Similarly we can compute the dual basis function for \eqref{phi-3d} in 3D.
We note that both Lagrange nodal basis and its dual basis are affine
  invariant.  That is, the Lagrange basis on the reference triangle
   is also the Lagrange basis on a general triangle after an affine mapping.
 For simplicity, we use the same notations $\phi_j$ and $\psi_j^{[2D]}$ for
   the nodal basis and the dual basis functions on the
     reference triangle and on a general triangle.

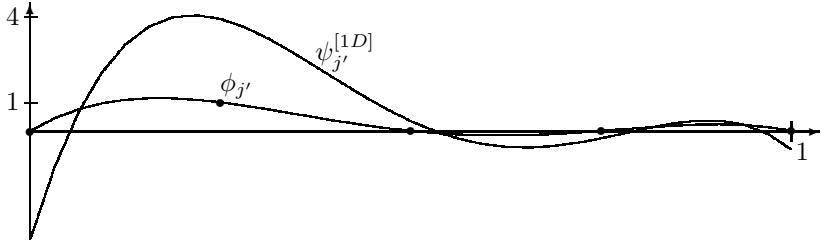
\begin{figure}[htb]\setlength\unitlength{0.9pt}\begin{center}
     \def\la{\vrule width.4pt height.4pt}
    \begin{picture}(350,100)(0,0)
    \put(20,5){\begin{picture}(80,80)(0,0)
  \put(0,40){\vector(1,0){332}} \put(322,28){$1$}
   \put(320,36){\line(0,1){8}} \put(80.00,  57.00){$\phi_{j'}$}
  \put(0,-5){\vector(0,1){99}}  \put(-10,85){$4$}\put(-10,49){$1$}
  \put(-2,88){\line(1,0){5}} \put(-2,52){\line(1,0){5}}
  \put(0.00,  40.00){\circle*{3}} \put(120.00,70){$\psi_{j'}^{[1D]}$}
  \put(80.00, 52.00){\circle*{3}}
  \put(160.00,  40.00){\circle*{3}}
  \put(240.00,  40.00){\circle*{3}}
  \put(320.00,  40.00){\circle*{3}}

 \multiput(   0.00,  40.00)(   0.295,   0.154){ 33}{\la}
 \multiput(  10.00,  45.22)(   0.312,   0.118){ 32}{\la}
 \multiput(  20.00,  49.02)(   0.323,   0.083){ 30}{\la}
 \multiput(  30.00,  51.60)(   0.330,   0.050){ 30}{\la}
 \multiput(  40.00,  53.12)(   0.333,   0.022){ 30}{\la}
 \multiput(  50.00,  53.78)(   0.333,  -0.002){ 30}{\la}
 \multiput(  60.00,  53.71)(   0.333,  -0.021){ 30}{\la}
 \multiput(  70.00,  53.07)(   0.331,  -0.036){ 30}{\la}
 \multiput(  80.00,  52.00)(   0.330,  -0.046){ 30}{\la}
 \multiput(  90.00,  50.61)(   0.329,  -0.052){ 30}{\la}
 \multiput( 100.00,  49.02)(   0.329,  -0.056){ 30}{\la}
 \multiput( 110.00,  47.33)(   0.329,  -0.056){ 30}{\la}
 \multiput( 120.00,  45.62)(   0.329,  -0.054){ 30}{\la}
 \multiput( 130.00,  43.98)(   0.330,  -0.050){ 30}{\la}
 \multiput( 140.00,  42.46)(   0.330,  -0.044){ 30}{\la}
 \multiput( 150.00,  41.12)(   0.331,  -0.037){ 30}{\la}
 \multiput( 160.00,  40.00)(   0.332,  -0.029){ 30}{\la}
 \multiput( 170.00,  39.13)(   0.333,  -0.020){ 30}{\la}
 \multiput( 180.00,  38.52)(   0.333,  -0.011){ 30}{\la}
 \multiput( 190.00,  38.19)(   0.333,  -0.002){ 30}{\la}
 \multiput( 200.00,  38.12)(   0.333,   0.006){ 30}{\la}
 \multiput( 210.00,  38.31)(   0.333,   0.013){ 30}{\la}
 \multiput( 220.00,  38.71)(   0.333,   0.019){ 30}{\la}
 \multiput( 230.00,  39.29)(   0.333,   0.024){ 30}{\la}
 \multiput( 240.00,  40.00)(   0.332,   0.026){ 30}{\la}
 \multiput( 250.00,  40.77)(   0.332,   0.025){ 30}{\la}
 \multiput( 260.00,  41.52)(   0.333,   0.022){ 30}{\la}
 \multiput( 270.00,  42.18)(   0.333,   0.015){ 30}{\la}
 \multiput( 280.00,  42.62)(   0.333,   0.005){ 30}{\la}
 \multiput( 290.00,  42.76)(   0.333,  -0.010){ 30}{\la}
 \multiput( 300.00,  42.46)(   0.332,  -0.029){ 30}{\la}
 \multiput( 310.00,  41.59)(   0.329,  -0.052){ 30}{\la}
      \end{picture}}
    \put(20,5){\begin{picture}(80,80)(0,0)
 \multiput(   0.00,  -5.47)(   0.106,   0.316){ 94}{\la}
 \multiput(  10.00,  24.39)(   0.133,   0.306){ 75}{\la}
 \multiput(  20.00,  47.44)(   0.168,   0.288){ 59}{\la}
 \multiput(  30.00,  64.51)(   0.214,   0.255){ 46}{\la}
 \multiput(  40.00,  76.44)(   0.266,   0.201){ 37}{\la}
 \multiput(  50.00,  83.97)(   0.311,   0.120){ 32}{\la}
 \multiput(  60.00,  87.81)(   0.332,   0.027){ 30}{\la}
 \multiput(  70.00,  88.63)(   0.329,  -0.052){ 30}{\la}
 \multiput(  80.00,  87.04)(   0.315,  -0.109){ 31}{\la}
 \multiput(  90.00,  83.59)(   0.301,  -0.144){ 33}{\la}
 \multiput( 100.00,  78.80)(   0.290,  -0.165){ 34}{\la}
 \multiput( 110.00,  73.12)(   0.284,  -0.175){ 35}{\la}
 \multiput( 120.00,  66.97)(   0.282,  -0.177){ 35}{\la}
 \multiput( 130.00,  60.70)(   0.285,  -0.173){ 35}{\la}
 \multiput( 140.00,  54.62)(   0.291,  -0.163){ 34}{\la}
 \multiput( 150.00,  49.00)(   0.299,  -0.148){ 33}{\la}
 \multiput( 160.00,  44.04)(   0.308,  -0.127){ 32}{\la}
 \multiput( 170.00,  39.91)(   0.317,  -0.102){ 31}{\la}
 \multiput( 180.00,  36.70)(   0.325,  -0.072){ 30}{\la}
 \multiput( 190.00,  34.49)(   0.331,  -0.040){ 30}{\la}
 \multiput( 200.00,  33.28)(   0.333,  -0.008){ 30}{\la}
 \multiput( 210.00,  33.03)(   0.333,   0.021){ 30}{\la}
 \multiput( 220.00,  33.65)(   0.330,   0.045){ 30}{\la}
 \multiput( 230.00,  35.00)(   0.328,   0.062){ 30}{\la}
 \multiput( 240.00,  36.88)(   0.326,   0.071){ 30}{\la}
 \multiput( 250.00,  39.07)(   0.326,   0.071){ 30}{\la}
 \multiput( 260.00,  41.27)(   0.328,   0.061){ 30}{\la}
 \multiput( 270.00,  43.13)(   0.331,   0.038){ 30}{\la}
 \multiput( 280.00,  44.27)(   0.333,  -0.001){ 30}{\la}
 \multiput( 290.00,  44.24)(   0.329,  -0.055){ 30}{\la}
 \multiput( 300.00,  42.56)(   0.311,  -0.120){ 32}{\la}
 \multiput( 310.00,  38.69)(   0.277,  -0.185){ 36}{\la}
          \end{picture}}

    \end{picture}\end{center}
\caption{
    A 1D nodal basis $\phi_{j'}$ \eqref{phi-1d}
    and its dual basis in  1D, $\psi_{j'}^{[1D]}$  \eqref{psi-1},
   $\int_0^1 \phi_{j'} \psi_{j'}^{[1D]}=1$.
   \label{psi-1-g} }
 \end{figure}

\comment{
p4:=a0 +a10*x + a01*y + a20*x^2 +a11*x*y + a02*y^2 +
    a30*x^3  + a21*x^2*y +a12*x*y^2 + a03*y^3 +
    a40*x^4  + a31*x^3*y  +a22*x^2*y^2 +a13*x*y^3 + a04*y^4:
ph2:=(x,y) -> x*(1-x-y)*(3/4-x-y)*(2/4-x-y):
phi2:=ph2(x,y)/ph2(1/4,0):
i2:=int(int(phi2*phi2,y=0..1-x),x=0..1):
  e1:=int(int(p4*phi2,y=0..1-x),x=0..1)=1:

f:=1: g:=int(int(f*phi2/i2,y=0..1-x),x=0..1):
  e2:=int(int(p4*(f-g*phi2),y=0..1-x),x=0..1)=0:
f:=x: g:=int(int(f*phi2/i2,y=0..1-x),x=0..1):
  e3:=int(int(p4*(f-g*phi2),y=0..1-x),x=0..1)=0:
f:=y: g:=int(int(f*phi2/i2,y=0..1-x),x=0..1):
  e4:=int(int(p4*(f-g*phi2),y=0..1-x),x=0..1)=0:
f:=x^2: g:=int(int(f*phi2/i2,y=0..1-x),x=0..1):
  e5:=int(int(p4*(f-g*phi2),y=0..1-x),x=0..1)=0:
f:=x*y: g:=int(int(f*phi2/i2,y=0..1-x),x=0..1):
  e6:=int(int(p4*(f-g*phi2),y=0..1-x),x=0..1)=0:
f:=y^2: g:=int(int(f*phi2/i2,y=0..1-x),x=0..1):
  e7:=int(int(p4*(f-g*phi2),y=0..1-x),x=0..1)=0:
f:=x^3: g:=int(int(f*phi2/i2,y=0..1-x),x=0..1):
  e8:=int(int(p4*(f-g*phi2),y=0..1-x),x=0..1)=0:
f:=x^2*y: g:=int(int(f*phi2/i2,y=0..1-x),x=0..1):
  e9:=int(int(p4*(f-g*phi2),y=0..1-x),x=0..1)=0:
f:=x*y^2: g:=int(int(f*phi2/i2,y=0..1-x),x=0..1):
  e10:=int(int(p4*(f-g*phi2),y=0..1-x),x=0..1)=0:
f:=y^3: g:=int(int(f*phi2/i2,y=0..1-x),x=0..1):
  e11:=int(int(p4*(f-g*phi2),y=0..1-x),x=0..1)=0:
f:=x^4: g:=int(int(f*phi2/i2,y=0..1-x),x=0..1):
  e12:=int(int(p4*(f-g*phi2),y=0..1-x),x=0..1)=0:
f:=x^3*y: g:=int(int(f*phi2/i2,y=0..1-x),x=0..1):
  e13:=int(int(p4*(f-g*phi2),y=0..1-x),x=0..1)=0:
f:=x^2*y^2: g:=int(int(f*phi2/i2,y=0..1-x),x=0..1):
  e14:=int(int(p4*(f-g*phi2),y=0..1-x),x=0..1)=0:
f:=x*y^3 : g:=int(int(f*phi2/i2,y=0..1-x),x=0..1):
  e15:=int(int(p4*(f-g*phi2),y=0..1-x),x=0..1)=0:
f:= y^4 : g:=int(int(f*phi2/i2,y=0..1-x),x=0..1):
  e16:=int(int(p4*(f-g*phi2),y=0..1-x),x=0..1)=0:
s:=solve({e1,e16,e2, e4,e5,e6,e7,e8,e9,e10,e11,e12,e13,e14,e15},
    {a0, a10, a01, a20,a11,a02, a30,a21,a12,a03, a40,a31,a22,a13,a04});
subs(s,p4);
}

We now define the Scott-Zhang interpolation operator:
   \a{ Q_0 : H^1 (\Omega) \to V_h, }
where $V_h$ is the $C^0$-$P_{k+2}$ finite element space defined in \eqref{Vh}.
$Q_0v$ is defined by the nodal values at three types nodes.

\begin{enumerate}
\item
For each corner node $c_j$ (shared by possibly more than two elements), we
  select one boundary $(d-1)$-dimensional simplex $C_j$ if $c_j$ is
    a boundary node, or any
   one  $(d-1)$-dimensional face simplex $C_j$ on which the node is,
    as $c_j$'s averaging patch.
C.f. Figure \ref{node}, the boundary node $c_j$ has a boundary edge $C_j$,
   while the corner node $c_{j''}$ can choose any one of four
   edges passing it, as its averaging patch.
In Figure \ref{node3d},  a corner node $c_j$ has a triangle $C_j$ as its
   averaging patch.
In both 2D and 3D, we use a same definition
   \an{\label{c-i} Q_0 v(c_j) = \int_{C_j} \psi_j^{[(d-1)D]} v(x) d x. }

\item
For each middle node $m_j$, the averaging patch  is the unique
    $(d-1)$-dimensional simplex $C_j$ containing $m_j$,
   see Figures \ref{node} and \ref{node3d}.
The interpolated nodal value is then determined by the unique solution
   of linear equations:
   \an{\label{m-i} \int_{C_j}\Big( \sum_{m_{j'} \in C_j}
       Q_0 v(m_{j'})\phi_{j'}(x)
    + \sum_{c_{j } \in C_j}
       Q_0 v(c_{j })\phi_{j }(x) - v(x) \Big)  p_{i}(x) dx = 0 }
  for all degree $(k+2-d)$ polynomials $p_{i}$ on ($d-1$)-simplex $C_j$,
    where $Q_0v(c_j)$ is defined in \eqref{c-i}.
In 2D, c.f. Figure \ref{node}, after we determine the nodal values
   at the two end points ($c_j$), we determine the middle-edge points'
    nodal value by \eqref{m-i}.

\item
  After determine all nodal values on the surface of each element, we
   define the interpolation inside the element:
  The nodal values at internal nodes ($Q_0(i_{j''})$)
    are determined by the unique
    solution of the following linear equations
   \an{\label{i-i} &\quad \int_{C_j}
     \Big( \sum_{i_{j''} \in C_j}Q_0 v(i_{j''})\phi_{j''}\Big)  p_{i} dx
    \\ \nonumber
    &= \int_{C_j} \Big(v - \sum_{m_{j'} \in C_j}
       Q_0 v(m_{j'})\phi_{j'}
    - \sum_{c_{j } \in C_j}
       Q_0 v(c_{j })\phi_{j }   \Big)  p_{i} dx }
  for all degree $(k+1-d)$ polynomials $p_{i}$ on $d$-simplex $C_j$.
\end{enumerate}
By \eqref{c-i}--\eqref{i-i}, the (refined) Scott-Zhang interpolation is
  \an{\label{s-z} Q_0 v = \sum_{x_j\in {\cal N}_h} Q_0v(x_j) \phi_j(x), }
   where $ {\cal N}_h $ is the set of all $C^0$-$P_{k+2}$ Lagrange nodes
    of triangulation ${\cal T}_h$.

\begin{remark} If all corner nodes have selected a same averaging patch $C_j$
   as the unique patch for the middle nodes on the patch, then
   the solution of \eqref{m-i} is the $L^2$-projection, i.e.,
    \an{\label{m-c-i}
          Q_0 v(m_{j'}) = \int_{C_j} \psi_{j'}^{[(d-1)D]} v(x) d x. }
    In fact, \eqref{m-c-i} is the definition of the original
       Scott-Zhang operator in \cite{Scott-Zhang}.
    In the same fashion,  if all patches of the face nodes are face
    $(d-1)$-simplexes of $C_j$, then the internal nodal values are
    exactly that of the $L^2$-projection on $C_j$, i.e., the solution of
     \eqref{i-i} satisfies
     \an{\label{i-c-i}  Q_0 v(i_{j''}) =
          \int_{C_j} \psi_{j''}^{[dD]} v(x) d x. }
    But if there are more than one triangle or tetrahedron in  ${\cal T}_h$,
      some $C_j$ must be from neighboring elements.
   So \eqref{m-c-i} and \eqref{i-c-i} can not be satisfied in general.
   Otherwise the Scott-Zhang operator would preserve mass of order $(k+2)$
    both on an element and on its faces.
 \end{remark}

\begin{lemma} The Scott-Zhang interpolation operator \eqref{s-z} is
  well-defined, i.e., the linear systems of equations
     \eqref{m-i} and \eqref{i-i} both have unique solutions.
  \end{lemma}

\comment{
restart;
p4:=a0 +a10*x +  a20*x^2 + a30*x^3+ a40*x^4  :
ph:=x-> (1/4-x)*(2/4-x)*(3/4-x)*(1-x): phi[1]:=ph(x)/ph(0):
ph:=x-> (0-x)*(2/4-x)*(3/4-x)*(1-x):   phi[2]:=ph(x)/ph(1/4):
ph:=x-> (0-x)*(1/4-x)*(3/4-x)*(1-x):   phi[3]:=ph(x)/ph(2/4):
ph:=x-> (0-x)*(1/4-x)*(2/4-x)*(1-x):   phi[4]:=ph(x)/ph(3/4):
ph:=x-> (0-x)*(1/4-x)*(2/4-x)*(3/4-x): phi[5]:=ph(x)/ph(1):
with(linalg):
a:=matrix(5,5):
for i from 1 to 5 do for j from 1 to 5 do
   a[i,j]:=int(phi[i]*phi[j],x=0..1):od: od:
b:=inverse(a);
for i from 1 to 5 do
psi[i]:=simplify(
 expand(b[1,i]*phi[1]+b[2,i]*phi[2]+b[3,i]*phi[3]+b[4,i]*phi[4]+b[5,i]*phi[5]));
  od;
as:=matrix(3,3): for i from 1 to 3 do for j from 1 to 3 do
   as[i,j]:=coeff(psi[j+1],x,i-1); od; od;
ai:=inverse(as);
as:=inverse(ai);
inverse(as);

 }

\begin{proof}
   For the linear system of equations \eqref{m-i}, we change the
     $P_{k+2-d}^{d-1}$ basis
    functions ($\{p_i=1,x,...,x^k\}$ when $d=2$, or
      $\{p_i=1,x,y,x^2,xy,...,y^{k-1}\}$ when $d=3$)
     uniquely as linear combinations of the Lagrange basis functions
     on the subinterval $C_j^s$ ($d=2$) or the subtriangle $C_j^s$
        ($d=3$), c.f.
    Figure \ref{1-basis}.
   \an{\label{b-n} p_i = \sum_j c_{i,j} \phi_j^s, \quad
             i=1,2,...,\dim (P_{k+2-d}^{d-1}). }
  The nodal basis functions on a simplex $C_j$ and its subsimplex
    $C_j^s$ differ by a bubble functions:
    \an{\label{l-b} \phi_j  = \phi_{j'}^s \frac{ b(\b x) }{ b (\b x_j)} ,
   } where $b(\b x)$ is the bubble functions assuming 0 on the boundary
   of $C_j$.  For example,  c.f. Figure \ref{1-basis}, when $d=2$ and
    $C_j=[0,1]$,
    \a{ b( x)&= x(1-x), \\
        \phi_1^s(x) &=\frac{(x-2/4)(x-3/4)}
               {(1/4-2/4)(1/4-3/4)}, \\
       \phi_2 (x) & = \phi_1^s(x) \frac{b(x)}{b(1/4)}. }

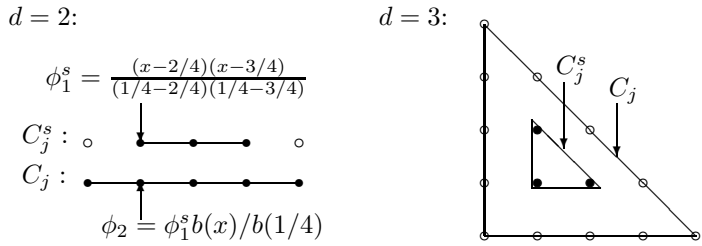
\begin{figure}[htb]\setlength\unitlength{1pt}\begin{center}
    \begin{picture}(280,90)(0,0)
    \put(30,15){\begin{picture}(80,80)(0,0)
   \put(-30,65){$d=2$:}
   \put(-25,5){$C_j:$}\put(0,5){\line(1,0){80}}
     \multiput(0,5)(20,0){5}{\circle*{3}}
   \put(-25,20){$C_j^s:$}\put(20,20){\line(1,0){40}}
    \multiput(20,20)(20,0){3}{\circle*{3}}\multiput(0,20)(80,0){2}{\circle{3}}
    \put(20,35){\vector(0,-1){15}}\put(-16,42){$\phi_1^s=\frac{(x-2/4)(x-3/4)}
               {(1/4-2/4)(1/4-3/4)}$}
   \put(5,-14){$\phi_2=\phi_1^s b(x)/b(1/4)$}\put(20,-9){\vector(0,1){15}}
          \end{picture}}
    \put(180,0){\begin{picture}(80,80)(0,0)
   \put(-40,80){$d=3$:}
    \put(0,0){\line(1,0){80}} \put(0,0){\line(0,1){80}}
     \put(80,0){\line(-1,1){80}}
    \put(18,18){\line(1,0){26}} \put(18,18){\line(0,1){26}}
     \put(44,18){\line(-1,1){26}}
    \multiput(0,0)(20,0){5}{\circle{3}} \multiput(0,20)(20,0){4}{\circle{3}}
    \multiput(0,40)(20,0){3}{\circle{3}} \multiput(0,60)(20,0){2}{\circle{3}}
    \multiput(0,80)(20,0){1}{\circle{3}}
     \multiput(20,20)(20,0){2}{\circle*{3}}
    \multiput(20,40)(20,0){1}{\circle*{3}}
     \put(30,58){\vector(0,-1){25}} \put(27,61){$C_j^s$}
     \put(50,50){\vector(0,-1){20}} \put(47,53){$C_j$}

      \end{picture}}

    \end{picture}\end{center}
\caption{ Lagrange nodal basis $\phi_j^s$
    on subinterval ($d=2$) or subtriangle ($d=3$), c.f., \eqref{l-b}.
   \label{1-basis} }
 \end{figure}

\comment{
restart;
b:=x-> x*(1-x):
ph:=x-> (2/4-x)*(3/4-x):   phi[1]:=ph(x)/ph(1/4): p[1]:=phi[1]*b(x)/b(1/4):
ph:=x-> (1/4-x)*(3/4-x):   phi[2]:=ph(x)/ph(2/4): p[2]:=phi[2]*b(x)/b(2/4):
ph:=x-> (1/4-x)*(2/4-x):   phi[3]:=ph(x)/ph(3/4): p[3]:=phi[3]*b(x)/b(3/4):
with(linalg):
a:=matrix(3,3):
for i from 1 to 3 do for j from 1 to 3 do
   a[i,j]:=int(p[i]*phi[j],x=0..1):od: od:
b:=inverse(a);inverse(b);

}
  By the change of basis, \eqref{b-n} and \eqref{l-b},
     the linear system \eqref{m-i} is equivalent
    to the following weighted-mass linear systems:
     \an{ \label{matrix}
       \sum_{m_{j'}\in C_j}
      Q_0 v(m_{j'})\int_{C_j}\phi_{j'} \phi_{i}^s dx
     = \int_{C_j} v \phi_{i}^s dx& - \sum_{j\in C_j}
                    Q_0 v(c_j) \int_{C_j} \phi_{j} \phi_{i}^s dx,\\
     &i=1,2,...,\dim (P_{k+2-d}^{d-1}).
   \nonumber }
 The coefficient matrix in \eqref{matrix} is the mass matrix
     on the subsimplex $C_j^s$ (c.f. Figure \ref{1-basis}) with
      a positive weight:
    \a{  a_{i,j'} &= \int_{C_j} \phi_{i}^s \phi_{j'} d\b x
               = \int_{C_j} \phi_{i}^s  \phi_{j}^s w(\b x) d\b x, }
     where \a{ w(\b x) = \frac{b(\b x)}{b(\b x_{j'})} >0 \quad
    \hbox{ in interior}(C_j). }
  As the Lagrange basis $\{ \phi_i^s \}$ (on the subsimplex)
     are linearly independent,  the mass (with weight) matrix
     in \eqref{matrix} is invertible,
   and the equivalent linear system \eqref{m-i} has a unique solution too.
   For example, when $d=2$ and $k=2$,  the coefficient matrix in \eqref{matrix}
   and its inverse are
   \a{ \p{152/315 & -16/63 &8/63 \\
          -4/21   & 18/35  & -4/21 \\
        8/63 & -16/63 & 152/315} ^{-1}
      =\p{ 2655/1024 & 75/64 & -225/1024 \\
           225/256 & 45/16  & 225/256 \\
           -225/1024 & 75/64 & 2655/1024 }.
     }
   By the same argument, lifting the space dimension by 1,
    we can show that \eqref{i-i} has a unique solution too.
   In fact, the system \eqref{i-i} when $d=2$ is the same system
      \eqref{m-i} with $d=3$ there.
\end{proof}

\begin{lemma}\label{l-2}
   If $d=2$, the Scott-Zhang interpolation operator \eqref{s-z}
   preserves the volume mass of order $k-1$ and the face mass of order $k$,
   i.e.,
    \an{\label{2d-v} \int_{T} (v - Q_0 v)p_i dx &=0
        \quad \forall T\in {\cal T}_h, \ p_i \in P_{k-1}^2(T), \\
    \label{2d-f} \int_{E} (v - Q_0 v)p_i dx &=0
        \quad \forall E\in {\cal E}_h, \ p_i \in P_{k}(E),}
   where ${\cal E}_h$ is the set of edges in triangulation ${\cal T}_h$.
  \end{lemma}
\begin{proof}
By the construction \eqref{m-i},
   we have
   \a{ \int_{E} (v - Q_0 v)p_i dx &=
    \int_{E} \Big( v - \sum_{m_{j'}\in E}
       Q_0v(m_{j'}) \phi_{j'}- \sum_{c_{j }\in E}
       Q_0v(c_{j }) \phi_{j }\Big)p_i dx=0. }
   That is \eqref{2d-f}.
   Here, because we have two missing dof (degrees of freedom) at the
    two end points of each edge,  the polynomial degree in mass preservation
    is reduced by two,  from $(k+2)$ to $k$.
  Similarly,  \eqref{2d-v} follows \eqref{i-i}.
    Here, the polynomial degree reduction is three as each triangle has
     three edges where the interpolation values are not free (not determined
     by \eqref{i-i}).
  \end{proof}

\begin{lemma} \label{l-3}
  If $d=3$, the Scott-Zhang interpolation operator \eqref{s-z}
   preserves the volume mass of order $k-2$ and the face mass of order $k-1$,
   i.e.,
    \an{\label{3d-v} \int_{T} (v - Q_0 v)p_i dx &=0
        \quad \forall T\in {\cal T}_h, \ p_i \in P_{k-2}^3(T), \\
    \label{3d-f} \int_{E} (v - Q_0 v)p_i dx &=0
        \quad \forall E\in {\cal E}_h, \ p_i \in P_{k-1}^2(E),}
   where ${\cal E}_h$ is the set of face triangles in
    the tetrahedral grid ${\cal T}_h$.
  \end{lemma}
\begin{proof}
  As each triangle $E$ has three edges, where the interpolation is
    not determined possibly by function value on neighboring triangles,
    we lose dof's on the three edges in the interpolation.
   That is, we lose three orders in face-mass conservation in \eqref{m-i}.
   \eqref{3d-f} is simply another expression of \eqref{m-i}, as
  in the proof of Lemma \ref{l-2}.
   By \eqref{i-i}, \eqref{3d-v} follows.
   Here the polynomial-degree deduction in mass conservation is 4, due
   to 4 face-triangles each tetrahedron.
  \end{proof}
\begin{remark} By Lemmas \ref{l-2} and \ref{l-3},
   the mass preservation on element and on faces is one order
   higher than the requirement \eqref{Q0}, when $d=2$.
   When $d=3$, \eqref{3d-v} and \eqref{3d-f} imply \eqref{Q0}.
 \end{remark}

\begin{theorem}\label{s-z-T} The Scott-Zhang interpolation operator
   \eqref{s-z} is of optimal order in approximation, i.e., when $k\ge -1$,
   \an{\label{o-a} \|v-Q_0v\|+h\|v-Q_0v\|_1& \le C h^{k+3}\|v\|_{k+3}
    \quad \forall v \in  H^{k+3}(\Omega). }
   Further, when $k\ge 0$,
    \an{\label{o-2}
    \big(\sum_{T\in {\cal T}_h} h^2 |v-Q_0v |_{H^2(T)}^2 \big)^{1/2}
        & \le C h^{k+3}
           \|v\|_{k+3}
    \quad \forall v \in  \cap H^{k+3}(\Omega). }
 \end{theorem}

\begin{proof} The Scott-Zhang operator preserves a degree $(k+2)$ polynomial
   on the star union of an element $T$,
   \a{  S_T = \cup_{\overline {T'}\cap T \ne \emptyset }  \overline{T'},\quad
         T, T' \in {\cal T}_h. }
  That is, \an{\label{preserve}
          Q_0 v = v \quad \forall v \in P_{k+2}^d(S_T), }
    when $k\ge -1$.
   \eqref{preserve} is shown in three steps.  First, by \eqref{c-i},
    when $v\in P_{k+2}$, the dual basis defines
    \an{\label{i-c-j} Q_0 v(c_j) = v(c_j), }
   at all corner nodes.
    In the second step, by \eqref{i-c-j} and \eqref{matrix},
       \eqref{m-i} holds for
       $p_i=\psi_{i}^s$
    too where $m_{i'}$ are all middle nodes on $C_j$.
    That is,
    \an{\label{i-m} \int_{C_j} \Big(v-\sum_{j\in C_j} Q_0(c_j)\phi_j
        -\sum_{m_{j'}} Q_0 v(m_{j'}) \phi_{j'}\Big) \phi_i^s dx
           &= 0.}
    By \eqref{i-c-j}, as $v\in P_{k+2}$,
    \a{ v-\sum_{j\in C_j} Q_0(c_j)\phi_j = v_c b(\b x) \quad
    \hbox{ for some } v_c \in P_{k+2-d}^{d-1}, }
     where $b(\b x)$ is a bubble function, cf. \eqref{l-b}.
    For the middle node basis functions, we have also, c.f. \eqref{l-b},
       \a{ \phi_{j'}= \phi_{j''}^s b(\b x)/b(\b x_{j'})
    \quad \forall m_{j'}\in C_j. }
  Thus \eqref{i-m} implies,
     where $ \sum_{m_{j'}} Q_0 v(m_{j'}) \phi_{j'}=v_m b(\b x)$
      for some $v_m\in P_{k+2-d}^{d-1}$,
      \a{  & \int_{C_j} (v_b-v_m) \phi_i^s b(\b x) d\b x  = 0, \\
      \Rightarrow \ & v_b-v_m\equiv 0 \quad \hbox{and } \
            v=Q_0 v \quad\hbox{ on } \ C_j.}
   Therefore, at all middle nodes,
      \an{\label{i-m-j} Q_0 v(m_{j'}) = v(m_{j'}). }
   In the third step,  by \eqref{i-i} and the same argument in the second step,
      \a{ Q_0 (i_{j''}) = v(i_{j''}), }
   at all internal nodes. Thus $Q_0 v=v\in P_{k+2}$.

   We then use the standard scaling argument (on the dual
        basis functions)  and the Sobolev inequality,
    as in Theorem 3.1 of \cite{Scott-Zhang},  it follows
    that
   \a{  |Q_0 v |_{H^1(T) } \le C \|v\|_{H^1(S_T) }
          \quad \forall v\in H^1 (\Omega).}
   The above stability result leads directly to the
    optimal-order approximation \eqref{o-a},  following the standard
      argument (i.e., by \eqref{preserve} and the existence of local
     Taylor polynomials, c.f. for example, \cite{Brenner-Scott}),
    as shown in Theorem 4.1 of \cite{Scott-Zhang}.
  We note again that the Scott-Zhang operator here is a refined version
    of  the Scott-Zhang operator in \cite{Scott-Zhang}.
   After showing the local preservation of $P_{k+2}$ polynomials above,
     the proof of the theorem is the same as that in \cite{Scott-Zhang}.
   \eqref{o-2}  is (4.4) in \cite{Scott-Zhang},
   with $p=q=2$, $m=2$, and $l=k+3$ there.

 \end{proof}

\end{document}